\documentclass[12pt, oneside,reqno]{amsart}
\usepackage{bbm}
\usepackage{mathtools}
\usepackage{amsmath, amsfonts, amssymb}
\usepackage{eucal}
\usepackage{mathrsfs}
\usepackage{latexsym}
\usepackage{cite}
\usepackage{xcolor}
\definecolor{refkey}{gray}{.45}
\definecolor{labelkey}{gray}{.45}

\usepackage{amsmath,amssymb}
\makeatletter
\newsavebox\myboxA
\newsavebox\myboxB
\newlength\mylenA

\newcommand*\xoverline[2][0.75]{%
    \sbox{\myboxA}{$\m@th#2$}%
    \setbox\myboxB\null
    \ht\myboxB=\ht\myboxA%
    \dp\myboxB=\dp\myboxA%
    \wd\myboxB=#1\wd\myboxA
    \sbox\myboxB{$\m@th\overline{\copy\myboxB}$}
    \setlength\mylenA{\the\wd\myboxA}
    \addtolength\mylenA{-\the\wd\myboxB}%
    \ifdim\wd\myboxB<\wd\myboxA%
       \rlap{\hskip 0.5\mylenA\usebox\myboxB}{\usebox\myboxA}%
    \else
        \hskip -0.5\mylenA\rlap{\usebox\myboxA}{\hskip 0.5\mylenA\usebox\myboxB}%
    \fi}
\makeatother

\newcommand{\rr}{\mathbb R}
\newcommand{\rbo}{\mathbb R}
\newcommand{\rn}{\mathbb R^n}

\newcommand{\sn}{S^{n-1}}

\newcommand{\kno}{\mathcal K^n_o}
\newcommand{\kne}{\mathcal K^n_e}
\newcommand{\sno}{\mathcal S^n_o}

\newcommand{\bx}{\pmb{x}}
\newcommand{\bu}{\pmb{\nu}}
\newcommand{\balpha}{\pmb{\alpha}}

\newcommand{\hm}{\mathcal H^{n-1}}
\newcommand{\conv}{\operatorname{conv}}
\newcommand{\spane}{\operatorname{span}}
\newcommand{\Beta}{\operatorname{B}}
\newcommand{\mcal}{\mathcal M}

\newcommand\wtilde[1]{\overset{\lower.4ex\hbox{$\scriptstyle \sim$}}{#1}}
\newcommand\wst[1]{\overset{\lower.5ex\hbox{$\scriptscriptstyle \sim$}}{#1}}
\newcommand{\blb}{\raise.3ex\hbox{$\scriptstyle \pmb \lbrack$}}
\newcommand{\sblb}{\raise.1ex\hbox{$\scriptscriptstyle \pmb \lbrack$}}
\newcommand{\brb}{\raise.3ex\hbox{$\scriptstyle \pmb \rbrack$}}
\newcommand{\sbrb}{\raise.1ex\hbox{$\scriptscriptstyle \pmb \rbrack$}}
\newcommand{\bla}{\raise.2ex\hbox{$\scriptstyle\pmb \langle$}}
\newcommand{\sbla}{\raise.1ex\hbox{$\scriptscriptstyle\pmb \langle$}}
\newcommand{\bra}{\raise.2ex\hbox{$\scriptstyle\pmb \rangle$}}
\newcommand{\sbra}{\raise.1ex\hbox{$\scriptscriptstyle\pmb \rangle$}}
\newcommand{\blrb}{\raise.3ex\hbox{$\scriptstyle \pmb | $}}
\newcommand{\sblrb}{\raise.1ex\hbox{$\scriptscriptstyle \pmb | $}}

\newcommand{\wt}{\wtilde}
\newcommand{\wtp}{\wtilde{+} }
\newcommand{\psum}{{+_{\negthinspace\kern-2pt p}}\,}
\newcommand{\qsum}[1]{{+_{\negthinspace\kern-2pt #1}}\,}
\newcommand{\dpsum}{{\tilde+_{\negthinspace\kern-1pt p}}\,}
\newcommand{\dqsum}[1]{{\tilde+_{\negthinspace\kern-1pt #1}}\,}
\newcommand{\lsub}[1]{\hskip -1.5pt\lower.5ex\hbox{$_{#1}$}}

\begin{document}

\title[Geometric measures and Minkowski problems]{Geometric measures in the dual Brunn-Minkowski theory and their associated Minkowski problems}

\author[Y. Huang]{Yong Huang}
\address{College of Mathematics and Econometrics,
Hunan University, Changsha, 410082, China}
\email{
huangyong@hnu.edu.cn}


\author[E. Lutwak]{Erwin Lutwak}
\address{Department of Mathematics, New York University Polytechnic School of Engineering, 6 Metrotech Center, Brooklyn, NY 11201, USA}
\email{lutwak@nyu.edu}

\author[D. Yang]{Deane Yang}
\address{Department of Mathematics, New York University Polytechnic School of Engineering, 6 Metrotech Center, Brooklyn, NY 11201, USA}
\email{deane.yang@nyu.edu}

\author[G. Zhang]{Gaoyong Zhang}
\address{Department of Mathematics, New York University Polytechnic School of Engineering, 6 Metrotech Center, Brooklyn, NY 11201, USA}
\email{gaoyong.zhang@nyu.edu}


\thanks{Research of the first author supported, in part by NSFC No.11371360; research of the other authors supported, in part, by
NSF Grant DMS-1312181.}


\maketitle

\numberwithin{equation}{section}

\newtheorem{theo}{Theorem}[section]
\newtheorem{coro}[theo]{Corollary}
\newtheorem{lemm}[theo]{Lemma}
\newtheorem{prop}[theo]{Proposition}
\newtheorem{conj}[theo]{Conjecture}
\newtheorem{exam}[theo]{Example}
\newtheorem{prob}[theo]{Problem}

\theoremstyle{definition}
\newtheorem{defi}[theo]{Definition}

\section{Introduction}

The Brunn-Minkowski theory (or the theory of mixed volumes) of
convex bodies, developed by Minkowski, Aleksandrov, Fenchel, et al.,
centers around the study of geometric functionals of convex bodies
as well as the differentials of these functionals.
The theory depends heavily on analytic tools such as
the cosine transform on the unit sphere
(a variant of the Fourier transform) and Monge-Amp\`ere type equations.
The fundamental geometric functionals in the Brunn-Minkowski theory
are the quermassintegrals (which include volume and surface area as special cases). The differentials
of volume, surface area and the other quermassintegrals are geometric measures called the area measures and (Federer's) curvature measures. These geometric measures are fundamental concepts in the Brunn-Minkowski theory.

A Minkowski problem is a characterization
problem for a geometric measure generated by convex bodies: It asks for necessary and sufficient conditions
in order that a given measure arises as the measure generated by a convex body. The solution of a Minkowski
problem, in general, amounts to solving a degenerated fully nonlinear partial differential equation.
The study of Minkowski problems has a
long history and strong influence on both the Brunn-Minkowski theory
and fully nonlinear partial differential equations, see \cite{LWY11} and \cite{S14}. Among the important
Minkowski problems in the classical Brunn-Minkowski theory are the classical Minkowski problem itself, the Aleksandrov problem, the Christoffel problem, and the Minkowski-Christoffel problem.

There are two extensions of the Brunn-Minkowski theory: the dual Brunn-Minkowski theory, which emerged in the mid-1970s, and the $L_p$ Brunn-Minkowski theory actively investigated since the 1990s
but dating back to the 1950s.  The important $L_p$ surface area measure and its associated Minkowski problem
in the $L_p$ Brunn-Minkowski theory were introduced in \cite{L93jdg}.
The logarithmic Minkowski problem and the centro-affine Minkowski problem are
unsolved singular cases, see \cite{BLYZ13jams} and \cite{CW06adv}.
The book \cite{S14} of Schneider presents a comprehensive account of the classical Brunn-Minkowski
theory and its recent developments, see Chapters 8 and 9 for Minkowski problems.

For the dual Brunn-Minkowski theory, the situation is quite different.
While, over the years, the ``duals" of many concepts and problems of the classical Brunn-Minkowski theory have been discovered and studied,
the duals of Federer's curvature measures and their associated
Minkowski problems within the dual Brunn-Minkowski theory have remained elusive. Behind this lay our inability to calculate the differentials of the dual quermassintegrals.
Since the revolutionary work of Aleksandrov in 1930s, the  nonlinear partial differential equations that arise within the classical Brunn-Minkowski theory and within the $L_p$ Brunn-Minkowski theory,
have done much to advance both theories. However, the intrinsic PDEs of the dual Brunn-Minkowski theory have had to wait a full 40 years after the birth of the dual theory to emerge. It was the elusive nature of the duals of Federer's curvature measures that kept these PDEs well hidden. As will be seen, the duals of Federer's curvature measures contain a number of surprises. Perhaps the biggest is that they connect known measures that were never imagined to be related. All this will be unveiled in the current work.

In the following, we first recall
the important geometric measures and their associated Minkowski problems in the classical Brunn-Minkowski theory and the $L_p$
Brunn-Minkowski theory.
Then we explain how the missing geometric measures in the dual Brunn-Minkowski theory can be naturally discovered and how their associated Minkowski problems will be investigated.

As will be shown, the notion of {\it dual curvature measures} arises naturally from the fundamental geometric functionals (the dual quermassintegrals) in the dual
Brunn-Minkowski theory.  Their associated Minkowski problem, will be called the  {\it dual Minkowski problem}. Amazingly,
both the logarithmic Minkowski problem as well as the Aleksandrov problem turn out to be special
cases of the new dual Minkowski problem. Existence conditions for the solution of the dual Minkowski problem in the symmetric case will be given.

\subsection{Geometric measures and their associated Minkowski problems in the Brunn-Minkowski theory}
The fundamental geometric functional for convex bodies in Euclidean $n$--space, $\rn$, is volume (Lebesgue
measure), denoted by $V$. The support function
$h_K:\sn \to {\mathbb R}$ of a compact convex $K \subset \rn$, is defined, for $v\in\sn$, the unit sphere, by
$h_K(v)=\max\{v\cdot x: x \in K\}$, where  $v\cdot x$ is the inner product of $v$ and $x$ in $\rn$.
For a continuous $f:\sn \to {\mathbb R}$, some small $\delta=\delta_K>0$, and $t \in(-\delta, \delta)$, define the $t$-perturbation of $K$ by $f$ by
\[
\blrb  K, f \blrb_t =\{x\in \rn : x\cdot v \le h_K(v)+tf(v),\ \text{for all}\  \ v\in\sn\}.
\]
This convex body is called the {\it Wulff shape} of $(K,f)$, with parameter $t$.
\smallskip

\noindent
{\bf Surface area measure, area measures, and curvature measures.}
Aleksandrov established the following
variational formula,
\begin{equation}\label{Alex-var-f}
\frac{d}{dt} V(\blrb K, f \blrb_t) \Big |_{t=0} = \int_{\sn} f(v)\, dS(K,v),
\end{equation}
where $S(K,\,\cdot\,)$ is the Borel measure on $\sn$ known as the classical {\it surface area measure of $K$}.
This formula suggests that the surface area measure
can be viewed as the differential of the volume functional. The total
measure of the surface area measure $S(K)=|S(K,\,\cdot\,)|$ is the ordinary surface area of $K$.
Aleksandrov's proof of \eqref{Alex-var-f} makes critical use of the Minkowski mixed-volume inequality
--- an inequality that is an extension of the classical isoperimetric inequality,
see Schneider \cite{S14}, Lemma 7.5.3, and \cite{HLYZ10}, Lemma 1. In this paper, we shall present the first proof of \eqref{Alex-var-f} that makes no use of mixed-volume inequalities.

The surface area measure of a convex body can be defined directly, for each Borel set $\eta \subset \sn$, by
\begin{equation}\label{s-a-m}
S(K, \eta) =\mathcal H^{n-1}(\nu_K^{-1}(\eta)),
\end{equation}
where $\mathcal H^{n-1}$ is $(n-1)$-dimensional
 Hausdorff measure. Here the Gauss map $\nu_K : \partial'\negthinspace K \to \sn$ is a function defined on all points of $\partial K$ that have a unique outer unit normal and is hence defined $\mathcal H^{n-1}$-a.e.\ on $\partial K$. (See \eqref{sim} for a precise definition.)
If one views the reciprocal Gauss curvature of a smooth convex body as a function of the outer unit normals of the body, then
surface area measure is the extension to arbitrary convex bodies (that are not necessarily smooth) of the reciprocal
 Gauss curvature. In fact,
 if $\partial K$ is of class $C^2$
 and has everywhere positive curvature, then the surface area measure has a positive density,
 \begin{equation}\label{det-hess}
 dS(K,v)/dv = \det(h_{ij}(v) + h_K(v)\delta_{ij}),
 \end{equation}
 where $(h_{ij})$ is the Hessian matrix of $h_K$ with respect to an orthonormal frame on $\sn$, 
 $\delta_{ij}$
 is the Kronecker delta, the determinant is precisely the reciprocal
 Gauss curvature of $\partial K$ at the point with outer unit normal $v$, and where the Radon-Nykodim derivative is with respect to spherical Lebesgue measure.

The {\it quermassintegrals},
are the principal geometric functionals in the Brunn-Minkowski theory.
These are the elementary mixed volumes which include volume, surface area,
and mean width. In differential geometry, the quermassintegrals
are the integrals of intermediate mean curvatures of closed smooth convex hypersurfaces.
In integral geometry, the quermassintegrals are the means of the projection areas
of convex bodies:
\begin{equation}\label{quermass}
W_{n-i}(K) = \frac{\omega_n}{\omega_i} \int_{G(n,i)} \text{vol}_i(K|\xi)\, d\xi,\quad i=1, \dots, n,
\end{equation}
where $\xi\in G(n,i)$, the Grassmann manifold  of $i$-dimensional
 subspaces in $\rn$, while $K|\xi$ is the image of the orthogonal projection of $K$ onto $\xi$, where $\text{vol}_i$ is just $\mathcal H^{i}$, or Lebesgue measure in $\xi$, and $\omega_i$ is the $i$-dimensional
 volume of the $i$-dimensional unit ball. The integration here is with respect to the rotation-invariant probability measure on $G(n,i)$.

Since $V=W_0$ and with $S(K,\,\cdot\,)=S_{n-1}(K,\,\cdot\,)$, it would be tempting to conjecture that
Aleksandrov's variational formula \eqref{Alex-var-f} can be extended to
quermassintegrals; i.e., for each continuous $f:\sn \to {\mathbb R}$,
\begin{equation}\label{i1.1}
\frac{d}{dt} W_{n-j-1}(\blrb K, f \blrb_t ) \Big |_{t=0} = \int_{\sn} f(v)\, dS_j(K,v), \qquad j=0, \dots, n-1.
\end{equation}
Unfortunately, this is only known for the very special case where $f$ is a support function and where the derivative is a right derivative.
The measures defined by \eqref{i1.1} (for the case where $f$ is a support function and the derivative is a right derivative) are called the {\it area measures} and were introduced by Fenchel \& Jessen and
Aleksandrov, see Schneider \cite{S14}, p. 214.
The proof of the variational formula \eqref{i1.1}, for the case where $f$ is a support function and the derivative is a right derivative, depends on
the Steiner formula for mixed volumes. The lack of a general variational formula for the quermassintegrals is one of the main obstacles in tackling the PDEs associated with the area measures.

In addition to the area measures of Aleksandrov and Fenchel \& Jessen, there exists another set of measures $\mathcal C_0(K,\cdot), \dots, \mathcal C_{n-1}(K,\cdot)$
called {\it curvature measures}, which were introduced by Federer \cite{F59} for sets of positive reach,
and are also closely related to the quermassintegrals. A direct treatment of curvature measures for convex bodies
was given by Schneider \cite{S78, S79}, see also \cite{S14}, p.~214.
If $K$ is a convex body in $\rn$ that contains
the origin in its interior, then each ray emanating from the origin intersects a unique point on
$\partial K$ and a unique point on the unit sphere $\sn$. This fact induces a bi-Lipschitz map $r_K : \sn \to  \partial K$.
The pullback of $r_K$ sends the curvature measure $\mathcal C_j(K,\cdot)$ on $\partial K$
to a measure $C_j(K, \cdot)$ on the unit sphere $\sn$, which is called the {\it $j$-th
curvature measure} of $K$.  The measure $C_0(K,\cdot)$ was first defined by Aleksandrov,
who called it the {\it integral curvature of $K$}, see \cite{Al4}. The total measures of
both area measures and curvature measures give the quermassintegrals:
\[
S_j(K,\sn) = C_j(K, \sn) = n W_{n-j}(K),
\]
for $j=0, 1, \dots, n-1$, see Schneider \cite{S14}, p. 213.
\smallskip

\noindent
{\bf Minkowski problems in the Brunn-Minkowski theory.}
One of the main problems in the Brunn-Minkowski theory is characterizing the
area and curvature measures. The well-known classical Minkowski problem is: {\it given a finite
Borel measure $\mu$ on $\sn$, what are the necessary and sufficient conditions so that
$\mu$ is the surface area measure $S(K,\,\cdot\,)$ of a convex body $K$ in $\rn$?}
The Minkowski problem was first studied by Minkowski \cite{Min1, Min2}, who
demonstrated both existence and uniqueness of solutions for the problem when the given measure is either discrete or has a continuous density. Aleksandrov \cite{Al2, Al3} and
Fenchel \& Jessen \cite{FJ} independently solved the problem in 1938 for arbitrary measures.
Their methods are variational and \eqref{Alex-var-f} is crucial
for transforming the Minkowski problem into an optimization problem.
Analytically,  the Minkowski problem is equivalent to solving a degenerated
Monge-Amp\`ere equation. Establishing the regularity of the solution
to the Minkowski problem is difficult and has led to a long series of influential works,
see for example, Lewy \cite{Le38},
Nirenberg \cite{N53},
Cheng and Yau \cite{ChengYau}, Pogorolov \cite{Pogobook}, Caffarelli \cite{Caf90Ann}, etc.

After solving the Minkowski problem, Aleksandrov went on to characterize his integral
curvature $C_0(K,\cdot)$, which is called the {\it Aleksandrov problem}. He was
able to solve it completely by using his {\it mapping lemma}, see \cite{Al4}.
Further work on the Aleksandrov problem from the PDE and the mass transport viewpoints
are due to Guan-Li \cite{GL97com} and Oliker \cite{Ol}.

Finding necessary and sufficient conditions so that a given measure is
the area measure $S_1(K,\cdot)$ of a convex body $K$ is the {\it Christoffel problem}.
Firey \cite{F68} and Berg \cite{Be69} solved the problem independently. See Pogorelov \cite{Po53}
for a partial result in the smooth case, and see Schneider \cite{S77} for a more explicit
solution in the polytope case,
Grinberg-Zhang \cite{GZ99} for an abbreviated  approach to Firey's and Berg's solution,
Goodey-Yaskin-Yaskina \cite{GYY11} for a Fourier transform approach,
and Schneider \cite{S14}, \S 8.3.2.
In general, characterizing the area measure $S_j(K,\cdot)$ is called
the {\it Minkowski-Christoffel problem}: {\it given an integer $1\le j\le n-1$ and a finite
Borel measure $\mu$ on $\sn$, what are the necessary and sufficient conditions so that
$\mu$ is the area measure $S_j(K,\cdot)$ of a convex body $K$ in $\rn$.}
The case where $j=1$ is the Christoffel problem, and the case where $j=n-1$ is the classical
Minkowski problem. For $1<j<n-1$,
it has been a long-standing open problem.
Important progress was made recently by Guan-Ma \cite{GM03inv}.
See Guan-Guan \cite{GG02Ann} for a variant of this problem.

Extending Aleksandrov's work on the integral curvature and characterizing
other curvature measures is also a major unsolved problem:
{\it given an integer $1\le j \le n-1$ and a finite
Borel measure $\mu$ on $\sn$, what are the necessary and sufficient conditions so that
$\mu$ is the curvature measure $C_j(K,\cdot)$ of a convex body $K$ in $\rn$.}
This is the {\it Minkowski problem for curvature measures} which can also be
called the {\it general Aleksandrov problem}. See Guan-Lin-Ma \cite{GLM09} and the recent work of Guan-Li-Li \cite{GLL12duke} on this problem.
\smallskip

\noindent
{\bf Cone volume measure and logarithmic Minkowski problem.}
In addition to the surface area measure of a convex body, another fundamental measure
associated with a convex body $K$ in $\rn$ that contains the origin in its interior
is the {\it cone volume measure} $V_K$, defined for the Borel set $\eta\subset\sn$,
by
\begin{equation}\label{c-v-m}
V_K(\eta) = \frac1n \int_{x\in \nu_K^{-1}(\eta)} x\cdot \nu_K(x)\, d\mathcal H^{n-1}(x)
= V(K\cap c(\eta)),
\end{equation}
which is the volume of the cone $K\cap c(\eta)$,
where $c(\eta)$ is the cone of rays from the origin so that
$\partial K \cap c(\eta) = \nu_K^{-1}(\eta)$ for the Borel set $\eta \subset \sn$.

A very important property of the cone volume measure is its $SL(n)$ invariance, or simply called
affine invariance.
The area and curvature measures are all $SO(n)$ invariant. The $SL(n)$ invariance
of the cone volume measure makes the measure a useful notion in the geometry of normed spaces,
see \cite{BGMN8, LYZ10jdg, LYZ10adv, N07, NR03, PW12}.
The Minkowski problem for the cone volume
measure is called the {\it logarithmic Minkowski problem}; It asks for necessary and sufficient conditions
for a given measure on the unit sphere to be the cone volume measure of a convex body. The existence part of the
logarithmic Minkowski problem has been solved recently for the case of even measures within the class of origin-symmetric convex bodies, see \cite{BLYZ13jams}.
A sufficient condition for discrete (but not-necessarily even) measures was given by Zhu \cite{Zu}.
It was shown in \cite{BLYZ12adv} that the solution to both the existence and uniqueness questions
for the logarithmic Minkowski problem for even measures would lead to a stronger Brunn-Minkowski inequality.
It was shown in \cite{BLYZ13jdg} that the
necessary and sufficient
conditions for the existence of a solution to the logarithmic Minkowski problem for even measures are identical to
the necessary and sufficient
conditions for the existence of an affine transformation that maps the given measure into one that is
isotropic.
The problem has strong connections with curvature flows, see Andrews \cite{And99inv, And03jams}.
\smallskip

\noindent
{\bf $L_p$ surface area measure and $L_p$ Minkowski problem.}
The $L_p$ Brunn-Minkowski theory is an extension of the classical Brunn-Minkowski theory,
see \cite{HS09jdg, Lud03, LR10, L93jdg, LYZ00jdg, LYZ00duke, LYZ05jdg, LZ97jdg, S14}.
The $L_p$ surface area measure, introduced in \cite{L93jdg}, is a fundamental notion in the $L_p$-theory.
For  fixed $p\in \mathbb R$, and a convex body $K$ in $\rn$
that contains the origin in its interior, the {\it $L_p$ surface area measure} $S^{(p)}(K,\,\cdot\,)$ of $K$
is a Borel measure on $\sn$, defined for a Borel set $\eta\subset \sn$, by
\begin{equation}\label{p-s-a-m}
S^{(p)}(K,\eta) = \int_{x\in \nu_K^{-1}(\eta)} (x\cdot \nu_K(x))^{1-p}\, d\mathcal H^{n-1}(x).
\end{equation}
Both the surface area measure and the cone volume measure are the special cases $p=1$ and $p=0$,
respectively, of $L_p$ surface area measure.
The $L_p$ Minkowski problem, posed by Lutwak (see, e.g., \cite{L93jdg}),
asks for necessary and sufficient
conditions so that a given measure on the unit sphere be the $L_p$ surface area measure of a convex body;
see, e.g., \cite{Chen06adv, CW06adv, 
L93jdg, LO95jdg}.
The case of $p=1$ is the classical Minkowski problem, the case of $p=0$ is the logarithmic Minkowski
problem (see \cite{BLYZ13jams}), and the case of $p=-n$ is the centro-affine Minkowski problem
(see Chou-Wang \cite{CW06adv}, Lu-Wang \cite{LW13jde}, and Zhu \cite{Zu2}). The solution to the $L_p$ Minkowski
problem has proven to be a critical tool in
establishing sharp affine Sobolev inequalities via affine isoperimetric inequalities, see
\cite{CLYZ, HS09jfa, LYZ00jdg, LYZ02jdg, 
Z99jdg}.

\medskip

\subsection{Geometric measures and their associated Minkowski problems in the dual Brunn-Minkowski theory}
A theory analogous to the theory of mixed volumes was introduced in 1970s in \cite{L75paci}.
It demonstrates a remarkable duality in convex geometry, and thus is called the {\it theory
of dual mixed volumes}, or the {\it dual Brunn-Minkowski theory}. The duality, as a guiding principle,
is conceptual in a heuristic sense and has motivated much investigation. A good explanation of
this conceptual duality is given in Schneider \cite{S14}, p. 507.
The aspect of the duality  between projections and cross-sections of convex bodies
is thoroughly discussed in Gardner \cite{G06book}.
The duality will be called the {\it conceptual duality} in convex geometry.

The main geometric functionals in the dual Brunn-Minkowski theory are the {\it dual quermassintegrals}.
The following integral geometric definition of the dual quermassintegrals, via the volume of the central sections of the body, shows their
amazing dual nature to the quermassintegrals defined in \eqref{quermass},
\begin{equation}\label{d-quermass}
\wt W_{n-i}(K) = \frac{\omega_n}{\omega_i}
 \int_{G(n,i)} \text{vol}_i(K\cap\xi)\, d\xi,\qquad i=1, \dots, n.
\end{equation}
The volume functional $V$ is both the quermassintegral $W_0$ and the
dual quermassintegral $\wt W_0$. Earlier investigations in the dual Brunn-Minkowski theory centered around finding
isoperimetric inequalities involving dual mixed volumes that mirrored those for mixed volumes, see
Schneider \cite{S14}, \S9.3 -- \S9.4, and Gardner \cite{G06book}.
It was shown in \cite{Z99tams} that the fundamental kinematic formula for quermassintegrals
in integral geometry has a dual version for dual quermassintegrals.

Exciting developments in the dual Brunn-Minkowski theory began in the late 1980s
because of the duality between projection bodies and intersection bodies exhibited in \cite{L88adv}.
The study of central sections of convex bodies
by way of
intersection bodies and the  Busemann-Petty
problem has attracted extensive attention in convex geometry,
see, for example, \cite{Bo91, G94annals, GKS99, L88adv, Z99annals}, and see \cite{G06book, K05}
for additional references.
Some of these works bring techniques from harmonic analysis, in particular, Radon transforms and the Fourier transform,
into the dual Brunn-Minkowski theory, see \cite{G06book, K05}. This is similar to the applications
of cosine transform to the study of projection bodies and the Shephard problem in the Brunn-Minkowski theory (see \cite{S14}, \S10.11 and
\cite{G06book}, \S4.2). However, the Busemann-Petty problem is far more interesting and is a problem whose isomorphic version
is still a major open problem in asymptotic convex geometric analysis.

There were important areas where progress in the dual theory lagged that of the classical theory.
Extending Aleksandrov's variational formula \eqref{Alex-var-f}  to the dual
quermassintegrals is one of the main such challenges. It is critically needed in order to discover the duals of the
geometric measures in the classical theory. One purpose of this work is to solve this problem
and thus to add key elements to the conceptual duality
of the Brunn-Minkowski theory and the dual Brunn-Minkowski theory. The main concepts to be introduced are {\it dual curvature measures}.
\smallskip

\noindent
{\bf Dual curvature measures.}
For each convex body $K$ in $\rn$ that contains the origin in its interior,
we construct explicitly a set of geometric measures $\wt C_0(K,\cdot),\ldots, \wt C_n(K,\cdot)$,
on $\sn$ associated with the dual quermassintegrals, with
\[
\wt C_j(K,\sn) = \wt W_{n-j}(K),
\]
for $ j=0,\dots, n$. These geometric measures can be viewed as the differentials of the dual quermassintegrals.

Our construction will show how these geometric measures, via conceptual duality, are the duals
of the curvature measures, and thus warrant being called the  {\it dual curvature measures} of $K$.
While the curvature measures of a convex body depend closely on
the body's boundary, its dual curvature measures depend more on the body's interior, but yet have
deep connections with classical concepts.
In the case $j=n$, the dual curvature measure $\wt C_n(K, \cdot)$ turns out to be the cone volume measure of $K$.
In the case $j=0$, the dual curvature measure $\wt C_0(K, \cdot)$ is Aleksandrov's integral curvature
of the polar body of $K$ divided by $n$.
When $K$ is a polytope, $\wt C_j(K,\cdot)$ is discrete and is concentrated on
the outer unit normals of the facets of $K$ with weights depending on the cones
that are the convex hull of the facet and the origin.
Dual area measures are also defined.
The new geometric measures we shall develop demonstrate yet again the amazing conceptual duality
between the dual Brunn-Minkowski theory and the Brunn-Minkowski theory.

We establish dual generalizations of Aleksandrov's variational formula \eqref{Alex-var-f}.
Suppose $K$ is a convex body in $\rn$ that contains the origin in its interior,
$f:\sn \to {\mathbb R}$ is a continuous function. For a sufficiently small $\delta>0$, define a family of {\it logarithmic Wulff shapes},
\[
\blb K,f \brb_t=\{x\in \rn : x\cdot v \le h_t(v),\ \text{for all } v\in \sn \},
\]
for each $t\in(-\delta,\delta)$, where $h_t(v)$, for $v\in\sn$, is given by
\[
\log h_t(v) = \log h_K(v) + t f(v) + o(t,v),
\]
where $\lim_{t\to 0} o(t,v)/t = 0$, uniformly in $v$.
The main formula to be presented is:
\smallskip

\noindent
{\bf Variational formula for dual quermassintegrals.} For $1 \le j \le n$, and each convex body $K$ that contains the origin in the interior, there exists a Borel measure $\wt C_j(K,\,\cdot\,)$ on $\sn$ such that
\begin{equation}\label{d-var-f}
\frac{d}{dt} \wt W_{n-j}(  \blb K,f \brb_t    )\Big|_{t=0} = j \int_{\sn} f(v)\, d\wt C_j(K,v),
\end{equation}
for each continuous $f:\sn \to {\mathbb R}$.

Obviously \eqref{d-var-f} demonstrates that the dual curvature measures are differentials of the dual
quermassintegrals. Clearly \eqref{d-var-f} is the dual of the conjectured variational formula \eqref{i1.1}.
Aleksandrov's variational formula \eqref{Alex-var-f} is the special case of $j=n$ of \eqref{d-var-f}. Thus,
our formula is a direct extension of Aleksandrov's variational formula for volume
to dual quermassintegrals.
Our approach and method of proof are very different from both Aleksandrov's proof of \eqref{Alex-var-f}
and from the proof of \eqref{i1.1} for the case where the function involved is a support function.

\smallskip
The main problem to be solved is the following characterization problem for the dual
curvature measures.
\smallskip

\noindent
{\bf Dual Minkowski problem for dual curvature measures.}
{\it Suppose $k$ is an integer such that $1\le k\le n$. If $\mu$ is a finite
Borel measure on $\sn$, find necessary and sufficient conditions on $\mu$ so that
it is the $k$-th dual curvature measure $\wt C_k(K,\cdot)$ of a convex body $K$ in $\rn$.}
\smallskip

This will be called the {\it dual Minkowski problem}.
For $k=n$, the dual Minkowski problem is just the logarithmic Minkowski problem. As will be shown, when the
measure $\mu$ has a density function $g:\sn \to {\mathbb R}$, the partial differential
equation that is the dual Minkowski problem is a Monge-Amp\`ere type equation on $\sn$:
\begin{equation}\label{0-1-equ}
\frac1n h|\nabla h|^{k-n} \det(h_{ij}+ h\delta_{ij}) = g.
\end{equation}
where $(h_{ij})$ is the Hessian matrix of the (unknown) function $h$ with respect to an orthonormal frame on $\sn$, and $\delta_{ij}$ is the Kronecker delta.

If $\frac1n h|\nabla h|^{k-n}$ were omitted in \eqref{0-1-equ}, then \eqref{0-1-equ} would become
the partial differential equation of the classical Minkowski problem.
If only the factor $|\nabla h|^{k-n}$ were
omitted, then equation \eqref{0-1-equ} would become the partial differential equation
associated with the logarithmic Minkowski problem.
The gradient component in \eqref{0-1-equ} significantly increases the difficulty of the problem when compared to the classical Minkowski
problem or
the logarithmic Minkowski problem.

In this paper, we treat the important symmetric case when the measure $\mu$ is even and the solution is within the class of origin-symmetric bodies. As will be shown, the existence of
solutions depends on how much of the measure's mass can be concentrated on great subspheres.

Let $\mu$ be a finite Borel measure on $\sn$,
and $1\le k\le n$.
We will say that the measure $\mu$ satisfies the {\it $k$-subspace mass inequality}, if
\[
\frac{\mu(\sn \cap \xi_{i})}{\mu(\sn)} < 1- \frac{k-1}{k} \frac{n-i}{n-1} ,
\]
for each $\xi_i\in G(n,i)$ and for each $i=1, \dots, n-1$.

\smallskip
The main theorem of the paper is the following
\smallskip

\noindent
{\bf Existence for the dual Minkowski problem.}
{\it Let $\mu$ be a finite even Borel measure on $\sn$, and $1\le k\le n$.
If the measure $\mu$ satisfies the $k$-subspace mass inequality,
then there exists an origin-symmetric convex body $K$ in $\rn$ such that
$\wt C_k(K,\cdot) = \mu.$}
\smallskip

The case of $k=n$ was proved in \cite{BLYZ13jams}. New ideas and more
delicate estimates are needed to prove the intermediate cases. We remark
that existence for the dual Minkowski problem is far easier to prove
for the special case where the given measure $\mu$ has a positive continuous density,
(where the subspace mass inequality is trivially satisfied). The
degenerated general case for measures is substantially more delicate and requires
far more powerful techniques to solve. The sufficient 1-subspace mass inequality is obviously
necessary for the case of $k=1$. The sufficient $n$-subspace
mass inequality is also necessary for the case $k=n$, except that certain
equality conditions must also be satisfied as well (see [9] for details). It would be of considerable interest to learn if the established sufficient conditions for other cases turn out to be necessary as well.

\section{Preliminaries}

\subsection{Basic concepts regarding convex bodies}
Schneider's book \cite{S14} is our standard reference for the basics regarding convex bodies. The books \cite{G06book, Gruberbook} are also good references.

Let $\rn$ denote $n$-dimensional Euclidean space. For $x\in\rn$, let $|x|=\sqrt{x\cdot x}$ be the Euclidean norm of $x$. For $x\in\rn \setminus \{0\}$, define $\xoverline{x}  \in \sn$ by
$\xoverline{x}  = x/|x|$. For a subset $E$ in $\rn \setminus \{0\}$, let $\bar E = \{\bar x : x\in E\}$.
The origin-centered unit ball $\{x\in\rn : |x|\le 1\}$ is always denoted by $B$, and its boundary by $\sn$.
Write $\omega_n$ for the volume of $B$ and recall that its surface area is  $n\omega_n$.

For the set of continuous functions defined on the unit sphere $\sn$ write $C(\sn)$, and for
$f\in C(\sn)$ write $\| f \|=\max_{v\in\sn} |f(v)|$. We shall view $C(\sn)$ as endowed with the topology induced by this {\it max-norm}. We write $C^+(\sn)$ for the set of strictly positive functions in $C(\sn)$, and $C_e^+(\sn)$ for the set of functions in $C^+(\sn)$ that are even.

If $K\subset\rn$ is compact and convex, the
support function
$h_K:\rn \to \rr$ of $K$ is defined by
$
h_K(x) = \max\{x\cdot y : y\in K\},
$
for $x\in\rn$. The
support function is convex and homogeneous of degree $1$. A compact convex subset of $\rn$ is uniquely determined by its support function.

Denote by $\mathcal K^n$ the space of compact convex sets in $\rn$ endowed with the {\it Hausdorff metric}; i.e. the distance between $K,L\in\mathcal K^n$
is
$\|h_K - h_L\|$.
By a {\it convex body} in $\rn$ we will always mean a compact convex set with nonempty interior.
Denote by $\kno$ the class of convex bodies  in $\rn$ that contain the origin in their
interiors, and denote by $\kne$ the class of origin-symmetric convex bodies in $\rn$.

Let $K\subset\rn$ be compact and star-shaped with respect to the origin.
The radial function $\rho_K:\rn\setminus \{0\} \to \rr$ is defined by
\[
\rho_K(x) = \max\{\lambda : \lambda x \in K\},
\]
for $x \neq 0$.
A compact star-shaped (about the origin) set is uniquely determined by its radial function on $\sn$.
Denote by $\mathcal S^n$ the set of compact star-shaped sets.
A star body is a compact star-shaped set with respect to the origin whose radial function
is continuous and positive.  If $K$ is a star body, then obviously
\[
\partial K =\{\rho_K(u)u : u\in \sn\}.
\]
Denote by $\sno$ the space of star bodies in $\rn$ endowed with the {\it radial metric}; i.e., the distance between $K,L\in\sno$, is
$
\|\rho_K - \rho_L\|.
$
Note that $\kno\subset\sno$ and that
on the space $\kno$ the Hausdorff metric and radial metric are equivalent, and thus
$\kno$ is a subspace of $\sno$.

If $K\in \kno$, then it is easily seen that
the radial function and the support function of $K$ are related by,
\begin{align}
h_K(v)        &= \max\nolimits_{u\in\sn}   (u\cdot v) \,\rho_K(u) , \label{2.1-1} \\
\shortintertext{for $v\in\sn$, and}
1/\rho_K(u) &=  \max\nolimits_{v\in\sn}  (u\cdot v)/h_K(v),    \label{2.1-2}
\shortintertext{for $u\in\sn$.} \nonumber
\end{align}

For a convex body $K \in \kno$, the
{\it polar body} $K^*$ of $K$ is a convex body in $\rn$ defined by
\[
K^* = \{ x\in \rn : x\cdot y \le 1, \  \text{for all }\  y\in K \}.
\]
From the definition of the polar body, we see that on $\rn\setminus\{0\}$,
\begin{equation}\label{polar-identity}
 \rho_K = 1/h_{K^*}\quad\text{and}\quad h_K= 1/\rho_{K^*}.
\end{equation}

For $K,L\subset\rn$ that are compact and convex, and real $a,b \ge 0$,
the {\it Minkowski combination}, $aK + bL\subset\rn$, is the compact, convex set defined by
\[
aK + bL = \{ax + by :\text{$x\in K$ and $y \in L$}\},
\]
and its support function is given by
\begin{equation}\label{minkowski-combination}
h_{aK+bL} = ah_K + bh_L.
\end{equation}

For real $t>0$, and a convex body $K$, let
$K_t=K+tB$, denote the {\it parallel body} of $K$. The volume of the parallel body $K_t$
is a polynomial in $t$, called the {\it Steiner polynomial},
\[
V(K_t) = \sum_{i=0}^n \binom ni  W_{n-i}(K) t^{n-i}.
\]
The coefficient $W_{n-i}(K)$ is called the $(n-i)$-th {\it quermassintegral} of $K$
which is precisely the geometric invariant defined in \eqref{quermass}.

\smallskip

For $K,L\subset\rn$ that are compact and star-shaped (with respect to the origin), and real $a,b \ge 0$,
the {\it radial combination}, $aK {\wtp} bL\subset\rn$, is the compact star-shaped set defined by
\[
aK{\wtp} bL =\{ax+by : \text{$x\in K$ and $y \in L$, while $x\cdot y =|x| |y|$}\}.
\]
Note that the condition $x\cdot y =|x| |y|$ means that either $y=\alpha x$ or $x=\alpha y$ for some $\alpha \ge0$.
The radial function of the radial combination of two star-shaped sets is the combination of their radial functions; i.e.,
\[
\rho_{aK\wst +bL} = a\rho_K + b\rho_L.
\]

For real $t>0$, and star body $K$, let
$\wt K_t = K{\wtp} tB$, denote the {\it dual parallel body} of $K$. The volume of
the dual parallel body $\wt K_t$
is a polynomial in $t$, called the {\it dual Steiner polynomial},
\[
V(\wt K_t) = \sum_{i=0}^n \binom ni  \wt W_{n-i}(K) t^{n-i}.
\]
The coefficient $\wt W_{n-i}(K)$ is the $(n-i)$-th {\it dual quermassintegral} of $K$
which is precisely the geometric invariant defined in \eqref{d-quermass}. For the $(n-i)$-th {\it dual quermassintegral} of $K$
we have the easily established integral representation
\begin{equation}\label{d-quer-i}
\wt W_{n-i}(K) = \frac1n \int_{\sn} \rho_K^i (u)\, du,
\end{equation}
where such integrals should always be interpreted as being with respect to spherical Lebesgue measure.

In view of the integral representation \eqref{d-quer-i}, the
dual quermassintegrals can be extended in an obvious manner:
For $q\in\rbo$, and a star body $K$,
the $(n-q)$-th dual quermassintegral $\wt W_{n-q}(K)$
is defined by
\begin{equation}\label{def-quermass-q}
\wt W_{n-q}(K) = \frac1n \int_{\sn} \rho_K^q (u)\, du.
\end{equation}

For real $q\neq 0$, define the {\it normalized dual quermassintegral} $\bar W_{n-q}(K)$ by
\begin{equation}\label{def-quermass-q-normalized}
\bar W_{n-q}(K) = \left(\frac1{n\omega_n}\int_{\sn} \rho_K^q(u)\, du\right)^\frac1q,
\end{equation}
and for $q=0$, by
\begin{equation}\label{def-quermass-q-normalized-zero}
\bar W_n(K) =\exp\left(\frac1{n\omega_n}\int_{\sn} \log\rho_K(u)\, du \right).
\end{equation}
It will be helpful to also adopt the following notation:
\begin{equation}\label{def-quermass-q-normalized-zero-v}
\text{$\wt V_q (K) = \wt W_{n-q}(K)$ \quad and \quad $\bar V_q (K) = \bar W_{n-q}(K)$,}
\end{equation}
called the {\it $q$-th dual volume} of $K$ and the {\it normalized $q$-th dual volume} of $K$, respectively. Note, in particular, the fact that
\begin{equation}\label{normalized-volume}
\bar V_n (K) = [V(K)/{\omega_n}]^{1/n}.
\end{equation}

\subsection{The radial Gauss map of a convex body}\label{2.2}
Let $K$ be a convex body in $\rn$. For each $v\in \sn$, the hyperplane
\begin{equation*}
H_K(v) = \{x \in \rn : x\cdot v = h_K(v)\}
\end{equation*}
is called the {\it supporting hyperplane to $K$ with unit normal $v$}.

For $\sigma\subset\partial K$, the {\it spherical image of $\sigma$} is defined by
\begin{equation*}
\bu_K(\sigma) = \{v\in\sn : x\in H_K(v)\ \text{for some}\ x\in\sigma \} \subset\sn.
\end{equation*}
For $\eta\subset\sn$, the {\it reverse spherical image of $\eta$} is defined by
\[
\bx_K(\eta)=\{ x\in \partial K : x\in H_K(v) \text{ for some } v\in \eta\} \subset \partial K.
\]

Let $\sigma_K\subset \partial K$ be the set consisting of all $x\in \partial K$, for which the set
$\bu_K(\{x\})$, which we frequently abbreviate as $\bu_K(x)$, contains more than a single element. It is well known that $\hm (\sigma_K) = 0$
(see p. 84 of Schneider \cite{S14}).
The function
\begin{equation}\label{sim}
\nu_K : \partial K \setminus \sigma_K \to \sn,
\end{equation}
defined by letting $\nu_K(x)$ be the unique element in $\bu_K(x)$, for each
$x \in \partial K \setminus \sigma_K $, is called the {\it spherical image map} of $K$ and is known to be continuous (see Lemma 2.2.12 of Schneider \cite{S14}). In the introduction, $\partial K \setminus \sigma_K$ was abbreviated as $\partial'\negthinspace K$, something we will often do. Note that from definition \eqref{s-a-m}, it follows immediately that for each continuous $g:\sn\to \rbo$,
\begin{equation}\label{s-a-m-1}
\int_{\partial'\negthinspace K} g(\nu_K(x))\, d\hm(x) = \int_{\sn} g(v)\,dS(K,v).
\end{equation}

The set $\eta_K\subset \sn$ consisting of all $v\in \sn$, for which the set $\bx_K(v)$ contains
more than a single element, is of $\mathcal H^{n-1}$-measure $0$ (see Theorem 2.2.11 of Schneider \cite{S14}). The function
\begin{equation}\label{rsim}
x_K : \sn  \setminus \eta_K \to \partial K,
\end{equation}
defined, for each $v \in \sn  \setminus \eta_K $,
by letting $x_K(v)$ be the unique element in $\bx_K(v)$, is called the {\it reverse spherical image map}. The vectors in $\sn  \setminus \eta_K$ are called the {\it regular normal vectors} of $K$. Thus, $v\in\sn$ is a regular normal vector of $K$ if and only if
$\partial K \cap H_K(v)$ consists of a single point. The function $x_K$
is well known to be continuous (see Lemma 2.2.12 of Schneider \cite{S14}).

For $K\in \kno$, define  the {\it radial map} of $K$,
\[
r_K : \sn \to \partial K\qquad\text{by}\qquad r_K(u) = \rho_K(u) u \in \partial K,
\]
for $u\in\sn$. Note that $r_K^{-1}: \partial K \to \sn$ is just the restriction of the map $\xoverline{\,\cdot\,}: \rn\setminus\{0\} \to\sn$ to $\partial K$.

For $\omega\subset\sn$, define the {\it radial Gauss image of $\omega$}  by
\[
\balpha_K(\omega) = \bu_K(r_K(\omega)) \subset \sn.
\]
Thus, for $u\in\sn$,
\begin{equation}\label{2.2-1}
\balpha_K(u) = \{ v\in \sn : r_K(u) \in H_K(v)\}.
\end{equation}

Define the {\it radial Gauss map} of the convex body $K\in\kno$
\[
\alpha_K : \sn\setminus \omega_K \to \sn\qquad\text{by}\qquad  \alpha_K = \nu_K \circ r_K,
\]
where $\omega_K = \xoverline{\sigma_K} = r_K^{-1} (\sigma_K)$. Since $r_K^{-1}= \xoverline{\,\cdot\,}$ is a bi-Lipschitz map between the spaces $\partial K$ and $\sn$ it follows that
$\omega_K$ has spherical Lebesgue measure $0$. Observe that if $u\in\sn\setminus \omega_K$, then $\balpha_K(u)$ contains only the element $\alpha_K(u)$.  Note that since both $\nu_K$ and $r_K$ are continuous, $\alpha_K$ is continuous.

For $\eta\subset\sn$, define the {\it reverse radial Gauss image} of $\eta$ by
\begin{equation}\label{alpha-star}
\balpha^*_K(\eta) = r^{-1}_K ( \bx_K(\eta) )= \xoverline{\bx_K(\eta)}.
\end{equation}
Thus,
\[
\pmb{\alpha}^*_K(\eta) =\{\xoverline{x} : \text{  $x\in\partial K$ where $x\in H_K(v)$ for some $v \in \eta$}\}.
\]
Define the {\it reverse radial Gauss map} of the convex body $K\in\kno$,
\begin{equation}\label{alpha-star-function}
\alpha^*_K : \sn\setminus \eta_K \to \sn,\qquad\text{by}\qquad  \alpha^*_K = r_K^{-1} \circ x_K.
\end{equation}
Note that since both $r_K^{-1}$ and $x_K$ are continuous, $\alpha^*_K $ is continuous.

Note for a subset $\eta \subset \sn$,
\begin{equation}\label{2.2-2}
\pmb\alpha^*_K(\eta) = \{ u\in \sn : r_K(u) \in H_K(v) \text{ for some } v\in \eta\}.
\end{equation}
For $u \in \sn$ and $\eta \subset \sn$, it is easily seen that
\begin{equation}\label{efg}
u\in \balpha^*_K(\eta)\quad \text{if and only if}\quad \balpha_K(u)\cap \eta \neq \varnothing.
\end{equation}
Thus,
\[
\text{$\balpha^*_K$ is monotone non-decreasing with respect to set inclusion.}
\]
If $\eta$ is the singleton containing only $v\in\sn$, then \eqref{efg} reduces to
\begin{equation}\label{reverse}
u\in \balpha^*_K(v)\quad \text{if and only if}\quad v\in\balpha_K(u).
\end{equation}
If $u\not\in\omega_K$, then $\balpha_K(u)=\{ \alpha_K(u) \}$, and \eqref{efg} becomes
\begin{equation}\label{rot}
u\in \balpha^*_K(\eta)\quad \text{if and only if}\quad \alpha_K(u)\in \eta,
\end{equation}
and hence \eqref{rot} holds for almost all $u\in\sn$, with respect to spherical Lebesgue measure.

The following lemma consists of a basic fact regarding the reverse radial Gauss map. This fact is Lemma 2.2.14 of Schneider \cite{S14}, an alternate proof of which is presented below.

\begin{lemm}\label{Lebesgue}
If $\eta \subset \sn$ is a Borel set, then $\balpha^*_K(\eta)=\xoverline{ \bx_K(\eta)} \subset\sn$ is spherical Lebesgue measurable.
\end{lemm}

\begin{proof}

The continuity of $\alpha_K$ assures that
the inverse image $\alpha_K^{-1}(\eta)$,
of the Borel set $\eta$ in $\sn$,
is a Borel set
in the space $\sn \setminus \omega_K$ with relative topology.
Since each Borel set in  $\sn \setminus \omega_K$ is just the restriction
 of a Borel set in $\sn$,
it follows that $\alpha_K^{-1}(\eta)$ is
the restriction of a Borel set in $\sn$ to  $\sn \setminus \omega_K$,
and is thus Lebesgue measurable in $\sn$ (because $\omega_K$
has Lebesgue measure 0).
Since $\balpha^*_K(\eta)$ and $\alpha_K^{-1}(\eta)$ differ
by a set of Lebesgue measure $0$, the set $\balpha^*_K(\eta)$ must be
Lebesgue measurable in $\sn$ as well.
\end{proof}

If $g : \sn \to \mathbb R$ is a Borel function,
then $g\circ \alpha_K$ is spherical Lebesgue measurable because it is just the composition of a Borel function $g$ and a continuous function $\alpha_K$ in $\sn\setminus \omega_K$ with $\omega_K$
 having Lebesgue measure 0. Moreover, if $g$ is a bounded Borel function, then $g\circ \alpha_K$ is spherical Lebesgue integrable. In particular, $g\circ \alpha_K$ is spherical Lebesgue integrable, for each continuous function $g : \sn \to \mathbb R$.

\begin{lemm}\label{converge}
Suppose
$K_i\in\kno$ with  $\lim_{i\to\infty}K_i = K_0\in \kno$.
Let $\omega=\cup_{i=0}^\infty \omega_{K_i}$ be the set (of $\hm$-measure $0$) off of which all of the $\alpha_{K_i}$ are defined. Then if $u_i\in\sn\setminus \omega$ are such that $\lim_{i\to\infty}u_i= u_0\in\sn\setminus \omega$, then
$\lim_{i\to\infty}\alpha_{K_i}(u_i) = \alpha_{K_0}(u_0)$.
\end{lemm}

\begin{proof}
Since the sequence of radial maps
$r_{K_i}$ converges to
$r_{K_0}$, uniformly,
$r_{K_i}(u_i)\to r_{K_0}(u_0)$.  Let
$x_i=r_{K_i}(u_i)$
and
$v_i=\nu_{K_i}(r_{K_i}(u_i))=\alpha_{K_i}(u_i)$.

Since $x_i\in\partial K_i$ and (since $u_i\not\in \omega_{K_i}$) the vector $v_i$ is the unique outer unit normal to the support hyperplane of $K_i$ at $x_i$. Thus, we have
\[
x_i\cdot v_i = h_{K_i}(v_i).
\]
Suppose that a subsequence of the unit vectors $v_i$ (which we again call $v_i$) converges to $v'\in\sn$. Since $h_{K_i}$ converges to
$h_{K_0}$, uniformly,
$h_{K_i}(v_i)\to h_{K_0}(v')$.  But this, together with $x_i\to x_0$ and $v_i\to v'$ gives
\begin{equation}\label{mm}
x_0 \cdot v' = h_{K_0}(v').
\end{equation}
But $x_0=r_{K_0}(u_0)$ is a boundary point of $K_0$, and since $u_0\not\in \omega_{K_0}$,
we conclude from \eqref{mm} that $v'$ must be the unique outer unit normal of $K_0$ at
$x_0=r_{K_0}(u_0)$. And hence, $v'=\nu_{K_0}(  r_{K_0}(u_0)  )=\alpha_{K_0}(u_0)$. Thus, all convergent subsequences of
$v_i=\alpha_{K_i}(u_i)$ converge to $\alpha_{K_0}(u_0)$.

Consider a subsequence of $\alpha_{K_i}(u_i)$. Since $\sn$ is compact, the subsequence has a subsequence that converges, and by the above, necessarily it converges to $\alpha_{K_0}(u_0)$. Thus, every subsequence of $\alpha_{K_i}(u_i)$ has a subsequence that converges to $\alpha_{K_0}(u_0)$.
\end{proof}

\begin{lemm}\label{abc}
If $\{\eta_j\}$ is a sequence of subsets on $\sn$, then
\[
\balpha^*_K\left(\cup_{j} \eta_j\right) = \cup_{j} \balpha^*_K(\eta_j).
\]
\end{lemm}
\begin{proof}
If $v\in \cup_j \eta_j$, then $v\in \eta_{j_1}$ for some $j_1$, and
$\balpha^*_K(v) \subseteq \balpha^*_K(\eta_{j_1}) \subseteq \cup_{j} \balpha^*_K(\eta_j)$.
Thus, $\balpha^*_K\left(\cup_{j} \eta_j\right) \subseteq \cup_{j} \balpha^*_K(\eta_j)$.
If $u \in \cup_{j} \balpha^*_K(\eta_j)$, then for some $j_2$, we have $u\in \balpha^*_K(\eta_{j_2})
\subseteq \balpha^*_K\left(\cup_{j} \eta_j\right)$. Thus,
$\balpha^*_K\left(\cup_{j} \eta_j\right) \supseteq \cup_{j} \balpha^*_K(\eta_j)$.
\end{proof}

\begin{lemm}\label{hij}
If $\{\eta_j\}$
is a sequence of pairwise disjoint sets in $\sn$, then
$\{\balpha^*_K(\eta_j) \setminus \omega_K\}$ is pairwise disjoint as well.
\end{lemm}
\begin{proof}
Suppose there exists a $u$ such that
$u \in \balpha^*_K(\eta_{j_1}) \setminus \omega_K$ and $u \in \balpha^*_K(\eta_{j_2}) \setminus \omega_K$. Since $u \notin \omega_K$ we know that $\balpha_K(u)$ is a singleton.
But \eqref{efg}, in conjunction with
$u \in \balpha^*_K(\eta_{j_1})$  and $u \in \balpha^*_K(\eta_{j_2})$, yields
 $\balpha_K(u) \cap \eta_{j_1} \neq \varnothing$ and $\balpha_K(u) \cap \eta_{j_2} \neq \varnothing$, which contradicts the fact that $\eta_{j_1} \cap \eta_{j_2} \neq \varnothing$, since $\balpha_K(u)$ is a singleton.
\end{proof}

The reverse radial Gauss image of a convex body and the radial Gauss image of its polar body are related.

\begin{lemm}\label{measure-zero-bold}
If $K\in\kno$, then
\[
\balpha^*_K(\eta) = \balpha_{K^*}(\eta),
\]
for each $\eta \subset \sn$.
\end{lemm}

\begin{proof}
If suffices to show that
\[
\balpha^*_K(v) = \balpha_{K^*}(v),
\]
for each $v\in \sn$.
Fix $v\in \sn$.
From \eqref{2.2-1}, we see that for $u \in \sn$,
\[
\text{$u \in \balpha_{K^*}(v)$ if and only if $H_{K^*}(u)$ is a support hyperplane at $\rho_{K^*}(v) v$,}
\]
that is,
\[
\text{$u \in \balpha_{K^*}(v)$ if and only if $h_{K^*}(u) = (u \cdot v)\,\rho_{K^*}(v)$,}
\]
and, by \eqref{polar-identity}, this is the case if and only if
\[
h_{K}(v) = (v \cdot u)\,\rho_{K}(u) = v \cdot r_{K}(u),
\]
or equivalently, using \eqref{2.2-1}, if and only if,
\[
v \in \balpha_{K}(u).
\]
But, from \eqref{reverse} we know that $v \in \balpha_{K}(u)$ if and only if $u\in \balpha^*_K(v)$.
\end{proof}

For almost all $v\in\sn$, we have $\balpha^*_K(v) = \{ \alpha^*_K(v)\}$, and for almost all $v\in\sn$, we have $\balpha_{K^*}(v) = \{ \alpha_{K^*}(v)\}$. These two facts combine to give:

\begin{lemm}\label{measure-zero}
If $K\in\kno$, then
\[
\alpha^*_K = \alpha_{K^*},
\]
almost everywhere on $\sn$, with respect to spherical Lebesgue measure.
\end{lemm}

\subsection{Wulff shapes and convex hulls}

Throughout,  $\Omega\subset\sn$ will denote a closed set that is assumed not to be contained in any closed hemisphere of $\sn$.
Let $h: \Omega \to (0, \infty)$ be continuous.
The {\it Wulff shape} $\blb h \brb \in \kno$, also known as the {\it Aleksandrov body}, determined by $h$ is the convex body defined by
\[
\blb h \brb = \{ x \in \rn : \text{$x\cdot v \le h(v)$ for all $v\in \Omega$} \}.
\]
Obviously, if $K\in\kno$,
\[
\blb h_K \brb = K.
\]

Let $\rho: \Omega \to (0, \infty)$ be continuous.
Since $\Omega \subset \sn$ is assumed to be closed, and $\rho$ is continuous,
$\{\rho(u)u : u\in \Omega\}$ is a compact set in $\rn$. Hence,
the convex hull $\bla \rho \bra$  generated by $\rho$,
\[
\bla \rho \bra  = \conv\{\rho(u)u : u\in\Omega\},
\]
is compact as well (see Schneider \cite{S14}, Theorem 1.1.11).
Since $\Omega$ is not contained in any closed hemisphere of $\sn$, the compact convex set $\bla \rho \bra$ contains the origin in its interior.
Obviously, if $K\in\kno$,
\begin{equation}\label{conv-of-body}
\bla \rho_K \bra = K.
\end{equation}

We shall make frequent use of the fact that,
\begin{equation}\label{support-convex-hull}
h_{\sbla \rho \sbra}(v) =
\max\nolimits_{u\in\Omega} (v\cdot u) \rho(u),
\end{equation}
for all $v\in\sn$.

\begin{lemm}
Suppose $\Omega \subset \sn$ is a closed subset of $\sn$ that is not contained in any closed hemisphere of $\sn$ and that $\rho : \Omega \to (0, \infty)$ is continuous.
If $v$ is a regular normal vector of $\bla \rho\bra$, then $\balpha_{\sbla \rho\sbra}^*(v) \subset \Omega$.
\end{lemm}
\begin{proof} By \eqref{support-convex-hull} there exists a $u_0\in\Omega$ such that
\[
h_{\sbla \rho \sbra}(v) =
(u_0\cdot v) \rho(u_0).
\]
This means that
\begin{equation}\label{t5t}
\rho(u_0)u_0 \in H_{\sbla \rho \sbra}(v)=\{x\in\rn : x \cdot v =h_{\sbla \rho \sbra}(v)\},
\end{equation}
and since clearly $\rho(u_0)u_0\in  \bla \rho\bra$, it follows, from \eqref{t5t} that $\rho(u_0)u_0\in  \partial \bla \rho\bra$ and $u_0\in \balpha_{\sbla \rho\sbra}^*(v)$. But $v$ is a regular normal vector of $\bla \rho\bra$, hence
\[
\balpha_{\sbla \rho\sbra}^*(v) =\{\alpha_{\sbla \rho\sbra}^*(v)   \}.
\]
We conclude that $\alpha_{\sbla \rho\sbra}^*(v)=u_0 \in\Omega$, which completes the proof.
\end{proof}

The Wulff shape of a function and the convex hull generated by its reciprocal are related.

\begin{lemm}\label{wulff-hull}
Let $\Omega\subset\sn$ be a closed set that is assumed not
to be contained in any closed hemisphere of $\sn$.
Let $h: \Omega \to (0,\infty)$ be continuous.
Then the Wulff shape $\blb h \brb$ determined by $h$ and
the convex hull $\bla 1/h \bra$ generated by the function $1/h$ are polar reciprocals of each other; i.e.,
\[
\blb h \brb^* = \bla 1/h \bra.
\]
\end{lemm}

\begin{proof} Let $\rho=1/h$.
The radial function $\rho_{\sblb h \sbrb}$ of $\blb h \brb$ is given by
\begin{align*}
\rho_{\sblb h \sbrb}(u) &=\max\{r>0 : ru \in \blb h \brb\} \\
&=\max\{r>0 : \text{$ru\cdot v \le h(v)$ for all $v\in \Omega$}\} \\
&=\max\{r>0 : r \max\nolimits_{v\in\Omega} (u\cdot v)/h(v) \le 1 \} \\
&=1/\big(\max\nolimits_{v\in\Omega} (u\cdot v)/h(v)\big),
\end{align*}
 for each $u\in \sn$.
Thus, since $\rho = 1/h$, we see that for $u\in \sn$,
\[
1/\rho_{\sblb h \sbrb}(u) =\max\nolimits_{v\in\Omega} (u\cdot v)/h(v) = \max\nolimits_{v\in\Omega} (u\cdot v) \rho(v) = h_{\sbla \rho \sbra}(u),
\]
by \eqref{support-convex-hull}. This and \eqref{polar-identity} give the desired identity.
\end{proof}

We recall Aleksandrov's convergence theorem for Wulff shapes (see Schneider
\cite{S14}, p. 412): If a sequence of continuous $h_i : \Omega \to (0,\infty)$ converges uniformly to $h : \Omega \to (0, \infty)$, then
the sequence of Wulff shapes $\blb h_i \brb$ converges to the Wulff shape $\blb h \brb$.

We will use the following convergence of convex hulls: If a sequence of positive continuous functions
$\rho_i : \Omega \to (0,\infty)$ converges uniformly to $\rho : \Omega \to (0, \infty)$, then
the sequence of convex hulls $\bla \rho_i \bra$ converges to the convex hull $\bla \rho \bra$.
Lemma \ref{wulff-hull}, together with Aleksandrov's convergence theorem for Wulff shapes, provides a quick proof.

Let $f: \Omega \to \rbo$ be continuous, and $\delta>0$.
Let $h_t : \Omega \to (0, \infty)$ be a continuous function defined for each $t\in(-\delta,\delta)$ by
\begin{equation}\label{hst}
\log h_t(v)= \log h(v) + t f(v) + o(t,v),
\end{equation}
where $o(t,v)$ is continuous in $v\in \Omega$ and $\lim_{t\to 0} o(t,v)/t = 0$, uniformly with respect to  $v\in \Omega$.
Denote by
\[
\blb h_t \brb = \{x\in\rn : \text{$x\cdot v \le h_t(v)$ for all $v\in \Omega$}\},
\]
the Wulff shape determined $h_t$. We shall call $\blb h_t \brb$
a {\it logarithmic family of Wulff shapes formed by $(h,f)$}. On occasion, we shall write $\blb h_t \brb$ as $\blb h,f,t\brb$, and if $h$ happens to be the support function of a convex body $K$ perhaps as $\blb K,f,t\brb$, or as $\blb K,f,o,t\brb$, if required for clarity.

\smallskip

Let $g: \Omega \to \rbo$ be continuous and $\delta>0$.
Let $\rho_t : \Omega \to (0,\infty)$ be a continuous function defined for each $t\in(-\delta,\delta)$ by
\begin{equation}\label{rst}
\log \rho_t(u)= \log \rho(u) + t g(u) + o(t,u),
\end{equation}
where again $o(t,u)$ is continuous in $u\in \Omega$ and $\lim_{t\to 0} o(t,u)/t = 0$, uniformly with respect to  $u\in \Omega$.
Denote by
\[
\bla \rho_{t} \bra = \conv \{\rho_{t}(u)u : u\in\sn \}
\]
the convex hull generated by $\rho_{t}$. We will call $\bla \rho_{t} \bra$
a {\it logarithmic family of convex hulls generated by $(\rho,g)$}.
On occasion, we shall write $\bla \rho_{t} \bra$ as $\bla \rho,g,t\bra$, and if $\rho$ happens to be the radial function of a convex body $K$ as $\bla K,g,t\bra$, or as $\bla K,g,o,t\bra$, if required for clarity.

\subsection{Two integral identities}

\begin{lemm}
Suppose $K\in \kno$ and $q\in \mathbb R$. Then for each bounded Lebesgue integrable $f: \sn \to \rbo$,
\begin{equation}\label{u-x-formula}
\int_{\sn} f(u) \rho_K(u)^q\, du
= \int_{\partial'\negthinspace K} f(\bar x )|x|^{q-n}\, x\cdot \nu_K(x)\, d\mathcal H^{n-1}(x).
\end{equation}
\end{lemm}

\begin{proof}
We only need to establish,
\begin{equation}\label{u-x-formula-1}
\int_{\sn} f(u) \rho_K(u)^n \, du
= \int_{\partial'\negthinspace K} f(\bar x )\, x\cdot \nu_K(x)\, d\mathcal H^{n-1}(x),
\end{equation}
because replacing $f$ with $f\rho_K^{q-n}$ in \eqref{u-x-formula-1} gives \eqref{u-x-formula}.

We begin by establishing \eqref{u-x-formula-1} for $C^1$-functions $f$.
To that end, define  $F:  \rn \setminus \{0\} \to \rbo$ by letting $F(x) = f(\bar x )$ for $x\neq 0$. Thus $F(x)$ is a $C^1$ homogeneous function of degree 0 in $\rn\setminus \{0\}$. The homogeneity of $F$ implies that
$x\cdot \nabla F(x)= 0$, and thus we have, $\text{div}(F(x)x) = n F(x)$.

Let $B_\delta\subset K$ be the ball of radius $\delta>0$ and centered at the origin.
Apply the divergence theorem
for sets of finite perimeter (see \cite{EG}, Section 5.8, Theorem 1) to ${K\setminus B_\delta} $, and get
\begin{align*}
n \int_{K\setminus B_\delta} F(x)\, dx &=
\int_{\partial'\negthinspace K} F(x)\, x\cdot \nu_K(x)\, d\mathcal H^{n-1}(x)
- \int_{\partial B_\delta} F(x)\, x\cdot \nu_{B_\delta}(x)\, d\mathcal H^{n-1}(x) \\
&= \int_{\partial'\negthinspace K} F(x)\, x\cdot \nu_K(x)\, d\mathcal H^{n-1}(x)
- \delta^n \int_{\sn} F(u)\, du,
\end{align*}
and hence,
\[
\int_{\partial'\negthinspace K} F(x)\, x\cdot \nu_K(x)\, d\mathcal H^{n-1}(x)
=n \int_{K} F(x)\, dx.
\]
Switching to polar coordinates give
\begin{align*}
\int_{K} F(x)\, dx
&=  \int_{\sn} \int_0^{\rho_K(u)} F(ru) r^{n-1} dr\,du\\
&=\frac1n \int_{\sn} F(u) \rho_K^n (u)\, du,
\end{align*}
which establishes \eqref{u-x-formula-1} for $C^1$-functions.

Since every continuous function on $\sn$ can be uniformly approximated by $C^1$ functions,
\eqref{u-x-formula-1} holds whenever $f$ is continuous.

Define the measure $\wt S_n$ on $\sn$ by
\begin{equation*}
\wt S_n(\omega) = \frac1n \int_{\omega} \rho_K^n (u)\, du.
\end{equation*}
for each Lebesgue measurable set $\omega\subset\sn$, and define the measure $V_{\partial K}$ on ${\partial'\negthinspace K}$ by
\begin{equation*}
V_{\partial K}(\sigma) =  \int_{\sigma} x\cdot \nu_K(x)\, d\mathcal H^{n-1}(x),
\end{equation*}
for each $\hm$-measurable $\sigma\subset{\partial'\negthinspace K}$.

It is easily seen that there exists $m_0, m_1, m_2>0$ such that
\begin{equation}\label{zxc}
\text{$\hm(r_K(\omega)) \le m_0\, \hm(\omega)$, $V_{\partial K}(\sigma) \le m_1 \, \hm(\sigma)$, and $\wt S_n(\omega) \le m_2\, \hm(\omega)$,}
\end{equation}
for every Lebesgue measurable set $\omega$ and $\hm$-measurable $\sigma$.

Suppose $f:\sn \to \rbo$ is a bounded integrable function; say $|f(u)|\le m$, for all $u\in\sn$.
Lusin's theorem followed by the Tietze's extension theorem, guarantees the existence of an open
subset $\omega_j \subset \sn$ and a continuous function $f_j:\sn \to \rbo$ so that $\mathcal H^{n-1}(\omega_j)<\frac1j$, while
$f=f_j$ on $\sn \setminus \omega_j$, with
$|f_j(u)|\le m$, for all $u\in\sn$.

Observe that
\[
\Big|\int_{\sn}  \negthinspace \negthinspace \negthinspace   (f(u) - f_j(u)) \rho_K^n(u)\, du\Big|
\le \Big|\int_{\sn \setminus \omega_j}  \negthinspace \negthinspace \negthinspace \negthinspace \negthinspace \negthinspace   (f(u) - f_j(u)) \rho_K^n(u)\, du\Big| + 2mn \wt S_n(\omega_j),
\]
where the integral on the right is $0$, and that
\begin{align*}
\Big|\int_{\partial'\negthinspace K} (f(\bar x) - f_j(\bar x))\, & x\cdot \nu_K(x)\, d\hm(x)\Big|\\
&\le \Big|\int_{\partial'\negthinspace K\setminus r_K(\omega_j)}  \negthinspace \negthinspace \negthinspace \negthinspace \negthinspace \negthinspace \negthinspace \negthinspace \negthinspace   (f(\bar x) - f_j(\bar x))\, x\cdot \nu_K(x) \, d\hm(x)\Big| + 2m V_{\partial K}(r_K(\omega_j)),
\end{align*}
where the integral on the right is $0$.

In light of \eqref{zxc}, the above allows us to establish \eqref{u-x-formula-1} for the bounded integrable function $f$, given that we had established \eqref{u-x-formula-1} for the continuous functions $f_j$.
\end{proof}

\begin{lemm}
Suppose that $K\in \kno$ is strictly convex, and $f:\sn\to \rr$, and $F:\partial K\to \rr$ are continuous, then
\begin{equation}\label{boundary-sphere}
\int_{\sn} f(v) F(\nabla h_K(v)) h_K(v)\, dS(K,v)
= \int_{\partial'\negthinspace K} x\cdot \nu_K(x)\,f(\nu_K(x)) F(x)  \,  d\hm(x),
\end{equation}
where the integral on the left is with respect to the surface area measure of $K$.
\end{lemm}

\begin{proof}
First observe that from the definition of the support function $h_K$ and the definition of $\nu_K$, it follows immediately that for $x\in \partial'\negthinspace K$,
\begin{equation}\label{yy1}
h_K(\nu_K(x) ) = x\cdot \nu_K(x).
\end{equation}

The assumption that $K$ is strictly convex implies that  $\nabla h_K$ always exists. But a convex function that's differentiable must be continuously differentiable and hence $\nabla h_K$ is continuous
on $\sn$. We shall use the fact that
\begin{equation}\label{yy2}
\text{$\nabla h_K \circ \nu_K : \partial'\negthinspace K \to \partial'\negthinspace K$\qquad\quad
is the identity map.}
\end{equation}
(See e.g., Schneider \cite{S14}, p.\ 47).

Now, from \eqref{yy1}, \eqref{yy2}, and \eqref{s-a-m-1}, we have
\begin{align}
\int_{\partial'\negthinspace K} x\cdot \nu_K(x)\,&f(\nu_K(x)) F(x)  \,  d\hm(x)\\
&=
\int_{\partial'\negthinspace K} h_K( \nu_K(x)) f(\nu_K(x)) F(\nabla h_K(\nu_K(x))  \,  d\hm(x)\\
&=
\int_{\sn} f(v) F(\nabla h_K(v)) h_K(v)\, dS(K,v).
\end{align}
\end{proof}

\section{Dual curvature measures}

To display the conceptual duality between the Brunn-Minkowski theory and
the dual Brunn-Minkowski theory, we first follow Schneider \cite{S14}, Chapter 4,
to briefly develop the classical area and curvature measures
for convex bodies in the Brunn-Minkowski theory. Then we introduce two new families of geometric measures: the dual curvature and dual area measures, in the dual Brunn-Minkowski theory.
While curvature and area measures can be viewed as differentials of the quermassintegrals,
dual curvature and dual area measures are viewed as differentials of the dual quermassintegrals.

\subsection{Curvature and area measures}

Let $K$ be a convex body in $\kno$.
For $x\notin K$, denote by $d(K,x)$ the distance from $x$ to $K$.
Define the metric projection map $p_K : \rn \setminus K \to \partial K$ so that
$p_K(x) \in \partial K$ is the unique point  satisfying
\[
d(K,x) = |x-p_K(x)|.
\]
Denote by $v_K:\rn \setminus K \to \sn$ the {\it outer unit normal vector of $\partial K$ at $p_K(x)$},
defined by
\[
v_K(x) = \frac{x - p_K(x)}{d(K,x)},
\]
for $x\in \rn \setminus K$.

For $t > 0$, and Borel sets $\omega\subset \sn$ and $\eta \subset \sn$, let
\begin{align*}
A_t(K, \omega) &= \Big\{x\in \rn : 0< d(K,x) \le t, \  p_K(x) \in r_K(\omega)\Big\},\\
B_t(K, \eta) &= \Big\{x\in \rn : 0< d(K,x) \le t, \  v_K(x) \in \eta\Big\},
\end{align*}
which are the so-called local parallel bodies of $K$.
There are Steiner-type formulas,
\begin{align*}
V(A_t(K,\omega)) &= \frac1n \sum_{i=0}^{n-1} \binom ni t^{n-i} C_i(K,\omega),\\
V(B_t(K,\eta)) &= \frac1n \sum_{i=0}^{n-1} \binom ni t^{n-i}  S_i(K,\eta),
\end{align*}
where $C_i(K,\cdot)$ is a Borel measure on $\sn$, called the {\it $i$-th curvature measure} of $K$,
and $S_i(K,\cdot)$ is a Borel measure on $\sn$, called the {\it $i$-th area measure} of $K$. For all this, see Schneider \cite{S14}, \S4.2.

Note that the classical curvature measures are defined on the boundary $\partial K$, which are
the image measures of $C_i(K,\cdot)$ under the radial map $r_K: \sn \to \partial K$.
Since, for $K\in \kno$, the radial map $r_K$ induces a strong equivalence between $\sn$ and
$\partial K$, one can always define the curvature measures equivalently on either of the two spaces.

The $(n-1)$-th area measure $S_{n-1}(K,\cdot)$ is the usual surface area measure $S(K,\cdot)$
which can be defined, for Borel $\eta\subset\sn$, directly by
\begin{equation}\label{temp1}
S_{n-1}(K, \eta) = \mathcal H^{n-1}(\bx_K(\eta)).
\end{equation}
The $(n - 1)$-th curvature measure $C_{n-1}(K, \cdot)$ on $\sn$ can be defined, for Borel $\omega\subset\sn$, by
\begin{equation}\label{temp2}
C_{n-1}(K, \omega) = \mathcal H^{n-1}(r_K(\omega)).
\end{equation}
From \eqref{temp1} and \eqref{temp2}, and the fact that $\balpha^*_K = r_K^{-1}\circ \bx_K$, we see that the $(n - 1)$-th curvature measure $C_{n-1}(K, \cdot)$ on $\sn$ and the $(n-1)$-th area measure
$S_{n-1}(K,\cdot)$ on $\sn$, are related by
\begin{equation}\label{temp3}
C_{n-1}(K,\balpha^*_K(\eta)) = S_{n-1}(K, \eta),
\end{equation}
for each Borel $\eta\subset\sn$.
See Schneider \cite{S14}, Theorem 4.2.3.

The 0-th area measure $S_0(K,\cdot)$ is just spherical Lebesgue measure on $\sn$; i.e.,
\[
S_0(K,\eta) = \mathcal H^{n-1}(\eta),
\]
for Borel $\eta\subset\sn$.
The 0-th curvature measure $C_0(K, \cdot)$ on $\sn$ can be defined, for Borel $\omega \subset \sn$,  by
\begin{equation}\label{i-curv-meas}
C_0(K, \omega) = \mathcal H^{n-1}(\balpha_K(\omega));
\end{equation}
that is, $C_0(K,\omega)$ is the spherical Lebesgue measure of $\balpha_K(\omega)$.
The 0-th curvature measure is also called the {\it integral curvature of $K$} which was first
defined by Aleksandrov. Obviously, \eqref{i-curv-meas} can be written as
\begin{equation}\label{0-cur-area-f}
C_0(K, \omega) = S_0(K, \balpha_K(\omega)).
\end{equation}
(See Schneider \cite{S14}, Theorem 4.2.3.)
If $K\in \kno$ happens to be strictly convex, then \eqref{0-cur-area-f} can be extended to
\begin{equation}\label{i-cur-area-f}
C_i(K, \omega) = S_i(K, \balpha_K(\omega)), \quad i=0, 1, \dots, n-1.
\end{equation}
(See Schneider \cite{S14}, Theorem 4.2.5.)

\subsection{Definition of dual curvature and dual area measures}
We first define the dual notions of the metric projection map $p_K$
and the distance function $d(K,\cdot)$.
Suppose $K\in\kno$.
Define the {\it radial projection map} $\wt p_K : \rn \setminus K \to \partial K$ by
\[
\wt p_K(x) = \rho_K(x)x=r_K(\xoverline{x}),
\]
for $x \in \rn \setminus K$. For $x \in \rn$, the {\it radial distance} $\wt d(K,x)$ of $x$ to $K$, is defined by
\[
\wt d(K,x) = \begin{cases} |x-\wt p_K(x)|  & x\notin K \\
0 &x\in K.
\end{cases}
\]
Let
\[
\wt v_K(x) = \xoverline{x}.
\]

For $t \ge 0$, a Lebesgue measurable set $\omega \subset \sn$,  and a Borel set $\eta \subset \sn$, define
\begin{align}
\wt A_t(K,\eta) &=\{x\in \rn : \text{ $0\le \wt d(K,x) \le t$ with $\wt p_K(x) \in \bx_K(\eta)$}\},
\label{d-cur-set} \\
\wt B_t(K, \omega) &= \{x\in \rn : \text{$0\le \wt d(K,x) \le t$ with $\wt v_K(x) \in \omega$}\},
\label{d-area-set}
\end{align}
to be the {\it local dual parallel bodies}. These local dual parallel bodies also have Steiner type
formulas as shown in the following theorem.

\begin{theo}\label{thm3.1}
Suppose $K\in \kno$. For $t \ge 0$, a Lebesgue measurable set $\omega \subset \sn$,  and a Borel set $\eta \subset \sn$,
\begin{align}
V(\wt A_t(K,\eta))&= \sum_{i=0}^{n} \binom ni t^{n-i} \wt C_i(K,\eta), \label{d-cur-S-f}\\
V(\wt B_t(K, \omega)) &= \sum_{i=0}^{n} \binom ni t^{n-i} \wt S_i(K,\omega), \label{d-area-S-f}
\end{align}
where $\wt C_i(K, \cdot)$ and $\wt S_i(K,\cdot)$ are  Borel
measures on $\sn$ given by
\begin{align}
\wt C_i(K, \eta)  &= \frac1n \int_{\balpha^*_K(\eta)} \rho_K^i(u)\, du, \label{d-cur-i} \\
\wt S_i(K,\omega) &= \frac1n \int_\omega \rho_K^i(u)\, du. \label{why-use-this}
\end{align}
\end{theo}

\begin{proof}
Write \eqref{d-area-set} as
\begin{equation}\label{bttilde}
\wt B_t(K, \omega) = \{x\in \rn : \text{$0\le |x| \le \rho_K(\xoverline{x}) + t$ with
$\xoverline{x} \in \omega$ }\}.
\end{equation}
Write $x=\rho u$, with $\rho\ge 0$ and $u\in \sn$, and
\begin{align*}
V(\wt B_t(K, \omega)) &= \int_{u\in \omega} \left(\int_0^{\rho_K(u)+t} \rho^{n-1}\, d\rho\right) du \\
&=\frac1n \int_{u\in \omega} (\rho_K(u) +t)^n \, du \\
&=\frac1n \sum_{i=0}^n \binom ni t^{n-i} \int_\omega \rho_K^i(u) \, du.
\end{align*}
This gives \eqref{d-area-S-f} and \eqref{why-use-this}.

In \eqref{d-cur-set}, the condition that $\wt p_K(x) \in \bx_K(\eta)$, or equivalently
$r_K(\xoverline{x}) \in \bx_K(\eta)$, is by \eqref{alpha-star}, the
same as $\xoverline{x} \in r^{-1}_K(\bx_K(\eta)) =\balpha^*_K(\eta)$.
Thus, \eqref{d-cur-set} can be written as
\begin{equation}\label{attilde}
\wt A_t(K,\eta) =\{ x\in \rn : \text{$0\le |x| \le \rho_K (\xoverline{x})  + t$ with
 $\xoverline{x} \in \balpha^*_K(\eta)$ }\}.
\end{equation}

Since $\eta \subset \sn$ is a Borel set, $\balpha^*_K(\eta)$ is a Lebesgue measurable subset of $\sn$ by Lemma \ref{Lebesgue}, and a glance at \eqref{bttilde} and \eqref{attilde} immediately gives
\[
\wt A_t(K,\eta) = \wt B_t(K, \balpha^*_K(\eta)).
\]
Now \eqref{d-area-S-f} yields,
\begin{align*}
V(\wt A_t(K,\eta))&= V(\wt B_t(K, \balpha^*_K(\eta))) \\
&= \sum_{i=0}^{n} \binom ni t^{n-i} \wt S_i(K, \balpha^*_K(\eta)),
\end{align*}
and by defining
\begin{equation}\label{def-c-t}
\wt C_i(K,\eta) = \wt S_i(K, \balpha^*_K(\eta)).
\end{equation}
we get both \eqref{d-cur-S-f} and \eqref{d-cur-i}.

Obviously, $\wt S_i(K,\cdot)$ is a Borel measure.
Note that since the integration in the integral representation of $\wt S_i(K,\cdot)$ is with respect to spherical Lebesgue measure, the measure $\wt S_i(K,\cdot)$ will assume the same value on sets that differ by a set of spherical Lebesgue measure $0$.

We now show that $\wt C_i(K,\cdot)$ is
a Borel measure as well. For the empty set $\varnothing$,
\[
\wt C_i(K, \varnothing)= \wt S_i(K, \balpha^*_K(\varnothing)) = \wt S_i(K, \varnothing)=0.
\]
Let $\{\eta_j\}$
be a sequence of pairwise disjoint Borel sets in $\sn$. From Lemma \ref{Lebesgue} and Lemma \ref{hij}, together with the fact that $\omega_K$ has spherical Lebesgue measure $0$,
we know that
$\{\balpha^*_K(\eta_j) \setminus \omega_K\}$ is
a sequence of pairwise disjoint Lebesgue measurable sets.
From Lemma \ref{abc}, the fact that $\omega_K$ has measure $0$, and the fact that the $\{\balpha^*_K(\eta_j)\setminus \omega_K\}$ are pairwise disjoint, together with \eqref{def-c-t}, we have
\begin{align*}
\wt C_i(K,\textstyle\bigcup_j \eta_j)
&= \wt S_i(K,\ \balpha^*_K(\textstyle\bigcup_j \eta_j)) \\
&=\wt S_i(K,\ \textstyle\bigcup_j \balpha^*_K(\eta_j)) \\
&=\wt S_i(K,\ (\textstyle\bigcup_j \balpha^*_K(\eta_j)) \setminus \omega_K) \\
&=\wt S_i(K,\ \textstyle\bigcup_j (\balpha^*_K(\eta_j)\setminus \omega_K)) \\
&=\textstyle\sum_j \wt S_i(K, \balpha^*_K(\eta_j)\setminus \omega_K) \\
&=\textstyle\sum_j \wt S_i(K, \balpha^*_K(\eta_j)) \\
&=\textstyle\sum_j \wt C_i(K, \eta_j).
\end{align*}
This shows that $\wt C_i(K, \cdot)$ is a Borel measure.
\end{proof}

We call the measure $\wt S_i(K,\cdot)$ the {\it $i$-th dual area measure of $K$}
and the measure $\wt C_i(K,\cdot)$ the {\it $i$-th dual curvature measure of $K$}.
From \eqref{d-quer-i}, \eqref{d-cur-i} and \eqref{why-use-this}, we see that the total measures
 of the $i$-th dual area measure and the $i$-th dual curvature measure
 are the $(n-i)$-th dual quermassintegral $\wt W_{n-i}(K)$; i.e.,
\begin{equation}\label{d-total-m}
\wt W_{n-i}(K)= \wt S_i(K,\sn) = \wt C_i(K,\sn).
\end{equation}

The integral representations \eqref{d-cur-i} and \eqref{why-use-this} show that the
dual curvature and dual area measures can be extended.

\begin{defi}
Suppose $K\in \kno$ and $q\in\rbo$. Define
the {\it $q$-th dual area measure} $\wt S_q(K,\cdot)$ by letting
\[
\wt S_q(K,\omega) = \frac1n \int_\omega \rho_K^q(u)\, du,
\]
for each Lebesgue measurable $\omega \subset \sn$,
and the {\it $q$-th dual curvature measure} $\wt C_q(K,\cdot)$ by letting
\begin{equation}\label{def-d-c-m}
\wt C_q(K,\eta) = \frac1n \int_{\balpha^*_K(\eta)} \rho_K^q(u)\, du = \frac1n \int_{\sn} \mathbbm{1}_{\balpha^*_K(\eta)}(u) \rho_K^q(u)\, du,
\end{equation}
for each Borel $\eta\subset\sn$.
\end{defi}

The verification that each $\wt C_q(K, \cdot)$ is a Borel measure is the same as for
the cases where $q =1, \dots, n$ as can be seen by examining the proof of this fact in Theorem \ref{thm3.1}.

Obviously, the total measures of the $q$-th dual curvature measure and the $q$-th
 dual area measure are the $(n-q)$-th dual quermassintegral; i.e.,
\begin{equation}\label{d-total-q}
\wt W_{n-q}(K)= \wt S_q(K,\sn) = \wt C_q(K,\sn).
\end{equation}

It follows immediately from their definitions that the $q$-th dual curvature measure of $K$ is the ``image measure" of the $q$-th
dual area measure of $K$ under $\balpha_K$; i.e.,
\begin{equation}
\wt C_q(K, \eta) = \wt S_q(K, \balpha^*_K(\eta)),
\end{equation}
for each Borel $\eta \subset \sn$.

\subsection{Dual curvature measures for special classes of convex bodies}

\begin{lemm} Suppose $K\in \kno$ and $q\in \mathbb R$. For
each function $g: \sn \to \mathbb R$ that is bounded and Borel,
\begin{equation}\label{cur-m-int}
\int_{\sn} g(v)\, d\wt C_q(K,v) = \frac1n\int_{\sn} g(\alpha_K(u)) \rho_K^q(u) \, du.
\end{equation}
\end{lemm}

A word of explanation is required. In the integral on the right in \eqref{cur-m-int}, the integration is with respect to spherical Lebesgue measure (recall that $\alpha_K$ is defined a.e.\ with respect to spherical Lebesgue measure).

\begin{proof}
Let $\phi$ be a simple function on $\sn$ given by
\[
\phi= \sum_i c_i \mathbbm{1}_{\eta_i}
\]
with $c_i\in\rbo$ and Borel $\eta_i\subset\sn$.
By using \eqref{def-d-c-m}, and \eqref{rot}, we get
\begin{align*}
\int_{\sn} \phi(v)\, d\wt C_q(K,v)
&= \int_{\sn} \sum_i c_i \mathbbm{1}_{\eta_i}(v) \, d\wt C_q(K,v)\\
&= \sum_i c_i \wt C_q(K, \eta_i)\\
&=\frac1n \int_{\sn} \sum_i c_i \mathbbm{1}_{\balpha^*_K(\eta_i)}(u) \rho_K^q(u)\, du \\
&=\frac1n \int_{\sn} \sum_i c_i \mathbbm{1}_{\eta_i}(\alpha_K(u)) \rho_K^q(u)\, du \\
&=\frac1n \int_{\sn} \phi(\alpha_K(u)) \rho_K^q(u)\, du.
\end{align*}
Choose a sequence of simple functions $\phi_k \to g$ uniformly. Then $\phi_k\circ \alpha_K$ converges
to $g\circ \alpha_K$ a.e.\ with respect to spherical Lebesgue measure.
Since $g$ is a Borel function and the radial Gauss map $\alpha_K$
is continuous, $g\circ \alpha_K$ is a Borel function. Thus, $g$ and
$g\circ \alpha_K$ are integrable since $g$ is bounded. Taking the limit $k \to \infty$
gives the desired \eqref{cur-m-int}.
\end{proof}

\begin{lemm}
Suppose $K\in \kno$ and $q\in \mathbb R$. For
each bounded Borel function $g: \sn \to \mathbb R$,
\begin{equation}\label{cur-m-int-2}
\int_{\sn} g(v)\, d\wt C_q(K,v)
= \frac1n\int_{\partial K} x\cdot \nu_K(x)\, g(\nu_K(x)) |x|^{q-n} \, d\hm(x).
\end{equation}
\end{lemm}

Take $f=g\circ \alpha_K$ in \eqref{u-x-formula}
and the desired  \eqref{cur-m-int-2} follows immediately from \eqref{cur-m-int}.

\begin{lemm}
Suppose $K\in \kno$ and $q\in \mathbb R$. For each Borel set $\eta\subset\sn$
\begin{equation}\label{cur-m-int-3}
\wt C_{q}(K,\eta)
=\frac 1n \int_{x\in \nu_K^{-1}(\eta)} x\cdot \nu_K(x)\, |x|^{q-n}  \, d\hm(x).
\end{equation}
\end{lemm}

Taking $g = \mathbbm{1}_\eta$ in \eqref{cur-m-int-2} immediately yields
\eqref{cur-m-int-3}.

We conclude with three observations regarding the dual curvature measures.
\smallskip

(i) Let $P\in \kno$ be a polytope with outer unit normals
$v_1,\ldots, v_m$. Let $\Delta_i$ be the cone that consists of all of the rays emanating from
the origin and passing through the facet of $P$ whose outer unit normal is $v_i$. Then, recalling that we would abbreviate $\balpha^*_P(\{v_i\})$ by $\balpha^*_P(v_i)$, we have
\begin{equation}\label{3.3-1}
\balpha^*_P(v_i) = \sn \cap \Delta_i.
\end{equation}
If $\eta\subset\sn$ is a Borel set such that $\{v_1,\ldots, v_m\} \cap \eta =\varnothing$,
then $\balpha^*_P(\eta)$ has spherical Lebesgue measure $0$.
Therefore, the dual curvature measure $\wt C_{q}(P,\cdot)$ is discrete
and is concentrated on $\{v_1, \dots, v_m\}$. From the definition of dual curvature measures \eqref{def-d-c-m},
and \eqref{3.3-1}, we see that
\begin{equation}\label{3.3-2}
\wt C_{q}(P,\cdot) = \sum_{i=1}^m c_i \delta_{v_i},
\end{equation}
where, $\delta_v$ denotes the delta measure concentrated at the point $v$ on $\sn$, and
\begin{equation}\label{3.3-3}
c_i = \frac1n \int_{\sn\cap \Delta_i} \rho_P(u)^q\, du.
\end{equation}

(ii) Suppose $K\in \kno$ is strictly convex. If $g:\sn\to\rbo$ is continuous, then \eqref{cur-m-int-2} and \eqref{boundary-sphere} give
\begin{align*}
\int_{\sn} g(v)\, d\wt C_q(K,v)
&= \frac1n\int_{\partial K} x\cdot \nu_K(x)\, g(\nu_K(x)) |x|^{q-n} d\hm(x) \\
&=\frac1n \int_{\sn} g(v) |\nabla h_K(v)|^{q-n} h_K(v)\, dS(K,v).
\end{align*}
This shows that
\begin{equation}\label{3.3-4}
d\wt C_{q}(K, \cdot) = \frac1n h_K |\nabla h_K|^{q-n}\, dS(K,\cdot).
\end{equation}

(iii) Suppose $K\in \kno$ has a $C^2$ boundary with everywhere positive curvature.
Since in this case $S(K,\cdot)$ is absolutely continuous
with respect to spherical Lebesgue measure, it follows that $\wt C_{q}(K,\cdot)$ is absolutely continuous
with respect to spherical Lebesgue measure, and from \eqref{3.3-4} and \eqref{det-hess}, we have
\begin{equation}\label{3.3-5}
\frac{d\wt C_{q}(K,v)}{dv}
= \frac1n  h_K(v) |\nabla h_K(v)|^{q-n} \det(h_{ij}(v) + h_K(v) \delta_{ij}),
\end{equation}
where $(h_{ij})$ denotes the Hessian matrix of $h_K$ with respect to an orthonormal frame on $\sn$.

\subsection{Properties of dual curvature measures}
The weak convergence of the $q$-th dual curvature
measure is critical and is contained in the following lemma.

\begin{lemm}
If $K_i\in\kno$ with  $K_i\to K_0\in \kno$, then
$\wt C_{q}(K_i, \cdot) \to \wt C_{q}(K_0,\cdot)$, weakly.
\end{lemm}

\begin{proof}  Suppose $g:\sn\to\rbo$ is continuous.
From \eqref{cur-m-int} we know that
\[
\int_{\sn} g(v)\, d\wt C_q(K_i,v) = \frac1n\int_{\sn} g(\alpha_{K_i}(u)) \rho_{K_i}^q(u) \, du,
\]
for all $i$.
Since $K_i \to K_0$, with respect to the Hausdorff metric, we know that
$\rho_{K_i} \to \rho_{K_0}$, uniformly, and using Lemma \ref{converge} that
$\alpha_{K_i} \to \alpha_K $, almost everywhere on $\sn$. Thus,
\[
\frac1n\int_{\sn} g(\alpha_{K_i}(u)) \rho_{K_i}^q(u) \, du
\    \longrightarrow   \
\frac1n\int_{\sn} g(\alpha_{K_0}(u)) \rho_{K_0}^q(u) \, du,
\]
from which it follows that $\wt C_q(K_i,\cdot) \to \wt C_q(K_0, \cdot)$, weakly.
\end{proof}

\begin{lemm}
If $K\in \kno$ and $q\in \mathbb R$, then the  dual curvature measure $\wt C_{q}(K,\cdot)$ is
absolutely continuous with respect to the surface area measure $S(K,\cdot)$.
\end{lemm}

\begin{proof}\label{ab-cont}
Suppose $\eta \subset \sn$ is such that $S(K,\eta)=0$, or equivalently, $\hm(\nu_K^{-1}(\eta))=0$.
 In this case, by \eqref{cur-m-int-3}, we conclude,
\[
\wt C_{q}(K,\eta)
=\frac 1n \int_{x\in \nu_K^{-1}(\eta)}|x|^{q-n} x\cdot \nu_K(x) \, d\hm(x) =0,
\]
since we are integrating over  a set of measure $0$.
Thus, the dual curvature measure $\wt C_{q}(K,\cdot)$ is
absolutely continuous with respect to the surface area measure $S(K,\cdot)$.
\end{proof}

The following lemma tells us that the $n$-th dual curvature measure of a convex body is the cone-volume measure of the body,
while the $0$-th dual curvature measure of the convex body is essentially Aleksandrov's integral curvature
of the polar of the body.

\begin{lemm}
If $K\in \kno$, then
\begin{align}
\wt C_n(K, \cdot) &= V_K, \label{dual-integral} \\
\wt C_0(K, \cdot) &= \frac1n C_0(K^*, \cdot). \label{dual-integral-2}
\end{align}
\end{lemm}

\begin{proof} Let $\eta \subset \sn$ be a Borel set.
From \eqref{cur-m-int-3}, with $q=n$, and \eqref{c-v-m}, we have
\[
\wt C_n(K, \eta) = \frac1n \int_{x\in\nu_K^{-1}(\eta)} x\cdot \nu_K(x)\, d\hm(x) = V_K(\eta),
\]
which establishes \eqref{dual-integral}.

From the definition of $0$-th dual curvature measure \eqref{d-cur-i}, with $i=0$, Lemma \ref{measure-zero-bold},
followed by \eqref{i-curv-meas}, we have
\[
\wt C_0(K,\eta) = \frac1n \hm(\balpha^*_K(\eta))
= \frac1n \hm(\balpha_{K^*}(\eta)) = \frac1n \, C_0(K^*,\eta),
\]
which gives \eqref{dual-integral-2}.
\end{proof}

The theory of valuations has witnessed explosive growth during the past two decades (see e.g. 
\cite{A1}--\negthinspace\negthinspace\cite{ A4}, \cite{Lud03}--\negthinspace\negthinspace\cite{LR10}, and \cite{S1}--\negthinspace\negthinspace\cite{S2}).     
Let $\mcal(\sn)$ denote the set of Borel measures on $\sn$. 
That the dual area measures are valuations whose range is $\mcal(\sn)$ is easily seen. But it turns out that the dual curvature measures are valuations (whose range is $\mcal(\sn)$) as well.
We now show that, for fixed index $q$, the functional that associates the body $K\in\kno$ with $\wt C_q(K, \cdot) \in \mcal(\sn)$ is a valuation.

\begin{lemm} For each real $q$,
the dual curvature measure $\wt C_q:\kno \to \mcal(\sn)$ is a valuation; i.e.,
if $K, L \in\kno$, are such that $K\cup L\in\kno$  then
\[
 \wt C_q(K, \cdot) + \wt C_q(L, \cdot) = \wt C_q(K\cap L, \cdot) + \wt C_q(K\cup L, \cdot).
 \]
\end{lemm}

\begin{proof}
Since for $Q\in\kno$, the function $r_Q: \sn \to \partial Q$ is a bijection, we have the following
disjoint partition of $\sn = \Omega_0 \cup \Omega_L \cup \Omega_K$, where
\begin{align*}
\Omega_0 &= r_K^{-1}(\partial K \cap \partial L) = r_L^{-1}(\partial K \cap \partial L)
=\{ u\in\sn: \rho_K(u)=\rho_L(u) \}, \\
\Omega_L &= r_K^{-1}(\partial K \cap \text{int} L) = r_L^{-1}((\rn\setminus K) \cap \partial L)
=\{ u\in\sn: \rho_K(u) < \rho_L(u) \},  \\
\Omega_K &= r_K^{-1}(\partial K \cap (\rn\setminus L)) = r_L^{-1}(\text{int} K \cap \partial L)
=\{ u\in\sn: \rho_K(u) > \rho_L(u) \}.
\end{align*}
Since $K\cup L$ is a convex body, for $\hm$-almost all $u \in \Omega_0$, we have
\begin{align*}
\rho_K(u)&=\rho_L(u) =\rho_{K\cap L}(u) = \rho_{K\cup L}(u), \\
\alpha_K(u)&=\alpha_L(u)=\alpha_{K\cap L}(u) = \alpha_{K\cup L}(u);
\end{align*}
For $\hm$-almost all $u\in \Omega_L $, we have
\begin{align*}
\rho_K(u)&=\rho_{K\cap L}(u), \ \rho_L(u) = \rho_{K\cup L}(u), \\
\alpha_K(u)&=\alpha_{K\cap L}(u), \ \alpha_L(u)=\alpha_{K\cup L}(u);
\end{align*}
For $\hm$-almost all $u\in \Omega_K$, we have
\begin{align*}
\rho_K(u)&=\rho_{K\cup L}(u), \ \rho_L(u) = \rho_{K\cap L}(u), \\
\alpha_K(u)&=\alpha_{K\cup L}(u), \ \alpha_L(u)=\alpha_{K\cap L}(u).
\end{align*}
From this it follows that if $g:\sn\to\rbo$ is continuous, then
\begin{align*}
\int_{\Omega_0} g(\alpha_K(u))\rho_K^q(u)\, du
&= \int_{\Omega_0} g(\alpha_{K\cap L}(u))\rho_{K\cap L}^q(u)\, du,\\
\int_{\Omega_L } g(\alpha_K(u))\rho_K^q(u)\, du
&= \int_{\Omega_L } g(\alpha_{K\cap L}(u))\rho_{K\cap L}^q(u)\, du,\\
\int_{\Omega_K} g(\alpha_K(u))\rho_K^q(u)\, du
&= \int_{\Omega_K} g(\alpha_{K\cup L}(u))\rho_{K\cup L}^q(u)\, du,\\
\int_{\Omega_0} g(\alpha_L(u))\rho_L^q(u)\, du
&= \int_{\Omega_0} g(\alpha_{K\cup L}(u))\rho_{K\cup L}^q(u)\, du,\\
\int_{\Omega_L } g(\alpha_L(u))\rho_L^q(u)\, du
&= \int_{\Omega_L } g(\alpha_{K\cup L}(u))\rho_{K\cup L}^q(u)\, du,\\
\int_{\Omega_K} g(\alpha_L(u))\rho_L^q(u)\, du
&= \int_{\Omega_K} g(\alpha_{K\cap L}(u))\rho_{K\cap L}^q(u)\, du.
\end{align*}
Summing up both sides of the integrals above gives
\begin{align*}
\int_{\sn} &g(\alpha_K(u))\rho_K^q(u)\, du + \int_{\sn} g(\alpha_L(u))\rho_L^q(u)\, du \\
&= \int_{\sn} g(\alpha_{K\cap L}(u))\rho_{K\cap L}^q(u)\, du
  + \int_{\sn} g(\alpha_{K\cup L}(u))\rho_{K\cup L}^q(u)\, du.
\end{align*}
Since this holds for each continuous $g$, we may appeal to \eqref{cur-m-int} to obtain the desired valuation property.
\end{proof}

\section{Variational formulas for the dual quermassintegrals}

When using the variational method to solve the Minkowski problem,
one of the crucial steps is to establish the variational formula
for volume which gives an integral of a continuous function on the unit
sphere integrated with respect to the surface area
measure. The variational formula is the key to transforming
the Minkowski problem into the Lagrange equation
of an optimization problem. Since the variational method needs to deal with
convex bodies that are not necessarily smooth, finding variational
formulas of geometric invariants of convex bodies is difficult.
In fact, for either quermassintegrals or dual quermassintegrals,
only the variational formula for one is known --- the volume.
This variational formula was established by Aleksandrov.

Let $K\in \kno$ and
let $f: \sn \to \rbo$ be continuous. For some $\delta>0$,
let $h_t: \sn \to (0,\infty)$ be defined for $v\in\sn$, and each $t\in(-\delta,\delta)$, by
\[
h_t(v)= h_K(v) + t f(v) + o(t,v),
\]
where $o(t,v)$ is continuous in $v\in \sn$ and $o(t,v)/t \to 0$, as $t\to 0$,
uniformly with respect to $v\in \sn$.

Let $\blb h_t \brb$ be the Wulff shape determined by $h_t$. Aleksandrov's variational
formula is
\[
\lim_{t\to 0} \frac{V(\blb h_t \brb) - V(K)}t = \int_{\sn} f(v)\, dS(K,v).
\]
The proof makes critical use of the Minkowski mixed-volume inequality.
Such a variational formula is not yet known for surface area or other quermassintegrals.

In this section, we shall take a completely different approach. Instead of considering
Wulff shapes, we consider convex hulls. We establish variational formulas
for all dual quermassintegrals. In particular, Aleksandrov's variational principle
will be established without using the Minkowski mixed-volume inequality.

Let $\Omega \subset \sn$ be a closed set that is not contained in any closed hemisphere of $\sn$.
Let $\rho_0 : \Omega \to (0,\infty)$
and $g: \Omega \to \mathbb R$ be continuous.
For some $\delta>0$, let $\rho_t : \Omega \to (0,\infty)$ be defined for $u\in\Omega$, and each
$t\in(-\delta,\delta)$, by
\begin{equation} \label{zx}
\log \rho_t(u)= \log \rho_0(u) + t g(u) + o(t,u),
 \end{equation}
where $o(t,u)$ is continuous in $u\in \Omega$ and $o(t,u)/t \to 0$, as $t\to 0$, uniformly with respect to $u\in \Omega$.
Recall that the logarithmic family of convex hulls $\bla \rho_{t} \bra$ of $(\rho_0, g)$,
indexed by $t\in (-\delta,\delta)$, is just the family of convex bodies ${\conv}\{\rho_t(u) u : u\in \Omega\}$, indexed by  $t\in(-\delta,\delta)$.

Since $\rho_t \to \rho_0$, uniformly on $\Omega$, for the associated bodies we have,
\begin{equation} \label{4.2}
 \bla \rho_{t} \bra \to \bla \rho_{0} \bra, \quad\text{as}\quad t\to 0.
 \end{equation}
We restate \eqref{support-convex-hull},
that for each $v\in \sn$,
\begin{equation} \label{4.3}
 h_{\sbla \rho_{t} \sbra}(v)=\max\nolimits_{u\in \Omega} \,(u\cdot v)\rho_t(u),
 \end{equation}
 for each $t\in(-\delta,\delta)$.

The following lemma shows that the support functions of a logarithmic family
of convex hulls are differentiable with respect to the variational variable.

\begin{lemm}\label{lemma4.1}
Suppose
$\Omega \subset \sn$ is a closed set that is not contained in any closed hemisphere of $\sn$.
Suppose $\rho_0 : \Omega \to (0,\infty)$, and $g: \Omega \to \mathbb R$ are continuous.
Let $\bla \rho_{t} \bra$ be a logarithmic family of convex hulls of $(\rho_0, g)$, then
\begin{equation}\label{radial-variation1}
\lim_{t\to 0} \frac{\log h_{\sbla \rho_{t} \sbra } (v) - \log h_{\sbla \rho_{0} \sbra}(v)} t = g(\alpha_{\sbla \rho_{0} \sbra}^{*}(v)),
\end{equation}
for all $v\in \sn\setminus \eta_{\sbla \rho_{0} \sbra}$; i.e., for all regular normals $v$ of ${\bla \rho_{0} \bra}$. Hence \eqref{radial-variation1} holds a.e. with respect to spherical Lebesgue measure. Moreover, there exists
a $\delta_0 > 0$ and an $M>0$ so that
\begin{equation}\label{4-bound}
|\log h_{\sbla \rho_{t} \sbra } (v) - \log h_{\sbla \rho_{0} \sbra}(v)| \le M |t|,
\end{equation}
for all $v\in \sn$ and all $t\in (-\delta_0, \delta_0)$.
\end{lemm}

\begin{proof}  Recall that $ \eta_{\sbla \rho_{0} \sbra}$ is the set of measure $0$, off of which, $\balpha^*_{\sbla \rho_{0} \sbra}$ is single valued and by
\eqref{alpha-star-function} the function
$\alpha^*_{\sbla \rho_{0} \sbra} : \sn\setminus \eta_{\sbla \rho_{0} \sbra} \to \sn$ is defined by $\alpha^*_{\sbla \rho_{0} \sbra} = r_{\sbla \rho_{0} \sbra}^{-1} \circ x_{\sbla \rho_{0} \sbra}$ or $\balpha^*_{\sbla \rho_{0} \sbra}(v) =\{\alpha^*_{\sbla \rho_{0} \sbra}(v)\}$.

Let $v\in \sn\setminus \eta_{\sbla \rho_{0} \sbra}$ be fixed throughout the proof.
From \eqref{4.3} we know that
there exist $u_t \in \Omega$ such that
\begin{equation}\label{m3h}
h_{\sbla \rho_{t} \sbra}(v)= (u_t\cdot v)\rho_t(u_t)\quad\text{while}\quad h_{\sbla \rho_{t} \sbra}(v) \ge (u\cdot v)\rho_t(u),
\end{equation}
for all $u\in\Omega$. Note that $u_t\cdot v >0$, for all $t$.

From \eqref{m3h}
we have
$h_{\sbla \rho_{0} \sbra}(v) =  (u_0\cdot v)\,\rho_0(u_0)$,
and hence
we know that
 $\rho_0(u_0)u_0 \in H_{\sbla \rho_{0} \sbra}(v)=\{x\in\rn : x \cdot v =h_{\sbla \rho_{0} \sbra}(v)\}$.
But $\rho_0(u_0)u_0 \in \bla \rho_{0} \bra$, so  $\rho_0(u_0)u_0 \in \partial \bla \rho_{0} \bra$,
and hence $u_0\in \balpha^*_{\sbla \rho_{0} \sbra}(v)=\{  \alpha^*_{\sbla \rho_{0} \sbra}(v)  \}$.  Thus
\begin{equation}  \label{m111}
u_0= \alpha^*_{\sbla \rho_{0} \sbra}(v).
\end{equation}

We first show that
\begin{equation}\label{hatter}
\lim\nolimits_{t\to 0} u_t = u_0
\end{equation}
where the $u_t$ come from \eqref{m3h}; i.e., are such that $h_{\sbla \rho_{t} \sbra}(v)= (u_t\cdot v)\rho_t(u_t)$. To see this consider a sequence $t_k\to 0$. The sequence $u_{t_k}$ in the compact set $\Omega$, has a convergent subsequence, which we also denote by $u_{t_k}$, such that
\[
u_{t_k}\rightarrow u'\in \Omega.
\]
Since $\rho_{t_k} \to \rho_0$, uniformly on $\Omega$,
\begin{equation}\label{4.12}
h_{\sbla \rho_{t_k} \sbra } (v) = (u_{t_k}\cdot v)\,\rho_{t_k}(u_{t_k}) \to (u' \cdot v)\,\rho_{0}(u').
\end{equation}
But $\rho_{t_k} \to \rho_0$, uniformly on $\Omega$ implies that $\bla \rho_{t_k} \bra \to \bla \rho_{0} \bra$, in $\kno$, and hence $h_{\sbla \rho_{t_k} \sbra }(v) \to h_{\sbla \rho_{0} \sbra }(v)$, which with \eqref{4.12} shows that $h_{\sbla \rho_{0} \sbra }(v)=(u' \cdot v)\,\rho_{0}(u')$. But this means that  $\rho_0(u')u' \in H_{\sbla \rho_{0} \sbra}(v)=\{x\in\rn : x \cdot v =h_{\sbla \rho_{0} \sbra}(v)\}$.
But $\rho_0(u')u' \in \bla \rho_{0} \bra$, so  $\rho_0(u')u' \in \partial \bla \rho_{0} \bra$,
and hence $u'\in \balpha^*_{\sbla \rho_{0} \sbra}(v)=\{  \alpha^*_{\sbla \rho_{0} \sbra}(v)  \}$. But from \eqref{m111}, we know $ \alpha^*_{\sbla \rho_{0} \sbra}(v)=u_0$, and thus $u'=u_0$. This establishes \eqref{hatter}.

From \eqref{2.1-1} we see that, for all $t$,
\begin{equation}\label{hat}
h_{\sbla \rho_{0} \sbra}(v) \ge (u_t\cdot v)\,\rho_{\sbla \rho_{0} \sbra}(u_t).
\end{equation}
Since $\bla \rho_{0} \bra={\conv}\{ \rho_0(u)u : u\in \Omega\}$,
we have
\begin{equation}  \label{m12h}
\rho_{\sbla \rho_{0} \sbra}(u_t) \ge \rho_0(u_t).
\end{equation}

From \eqref{m3h}, \eqref{hat}, \eqref{m12h},
and \eqref{zx}, we have
\begin{equation}  \label{m1h}
\begin{split}
 \log h_{\sbla \rho_{t} \sbra}(v)-\log h_{\sbla \rho_{0} \sbra}(v)&=\log \rho_t(u_t)+ \log (u_t\cdot v)-\log h_{\sbla \rho_{0} \sbra}(v) \\
 &\leq\log \rho_t(u_t)-\log \rho_{\sbla \rho_{0} \sbra}(u_t)\\
 &\leq\log \rho_t(u_t)-\log \rho_0(u_t)\\
 &= tg(u_t) + o(t,u_t).
 \end{split}
 \end{equation}

Since $\bla \rho_{t} \bra={\conv}\{\rho_t(u)u: u\in \Omega\}$, it follows that $\rho_{\sbla \rho_{t} \sbra}(u) \ge \rho_t(u)$, for all $u\in\Omega$.
From \eqref{m3h}, followed by
the fact that
$h_{\sbla \rho_{t} \sbra}(v) \ge  (u_0 \cdot v) \rho_t(u_0) $
from
\eqref{m3h}, and \eqref{zx}, we get
\begin{equation}\label{m2h}
\begin{split}
\log h_{\sbla \rho_{t} \sbra}(v)-\log h_{\sbla \rho_{0} \sbra}(v)
&=\log h_{\sbla \rho_{t} \sbra}(v)-\log \rho_0(u_0)-\log (u_0 \cdot v) \\
&\geq\log \rho_t(u_0)-\log \rho_0(u_0)\\
&= tg(u_0) + o(t,u_0).
 \end{split}
 \end{equation}

Let $M_0=\max_{u\in\Omega} |g(u)|$. Since $o(t,\cdot)/t \to 0$, as $t\to 0$, uniformly on $\Omega$, we may choose $\delta_0>0$ so that for all
$t\in (-\delta_0, \delta_0)$,
we have
$|o(t,\cdot)| \le |t|$ on $\Omega$. From \eqref{m1h} and  \eqref{m2h}, we immediately see that
\begin{equation}\label{add}
|\log h_{\sbla \rho_{t} \sbra } (v) - \log h_{\sbla \rho_{0} \sbra}(v)| \le (M_0+1) |t|.
 \end{equation}

 Combining \eqref{m1h} and  \eqref{m2h}, we have
 \[
 0\le \log h_{\sbla \rho_{t} \sbra}(v)-\log h_{\sbla \rho_{0} \sbra} (v) - t g(u_0) - o(t, u_0)
 \le t(g(u_t) - g(u_0)) + o(t,u_t) - o(t, u_0).
 \]
When $t>0$, this gives
\[
\frac{o(t, u_0)}t  \le
 \frac{\log h_{\sbla \rho_{t} \sbra}(v)-\log h_{\sbla \rho_{0} \sbra} (v)}t - g(u_0)
 \le g(u_t) - g(u_0) + \frac{o(t,u_t)}t.
\]
From \eqref{hatter} and the continuity of $g$, we can conclude that
\begin{equation}\label{limup}
\lim_{t \to 0^+}  \frac{\log h_{\sbla \rho_{t} \sbra}(v)-\log h_{\sbla \rho_{0} \sbra} (v)}t =  g(u_0).
\end{equation}
On the other hand, when $t<0$, we have
\[
\frac{o(t, u_0)}t
\ge \frac{\log h_{\sbla \rho_{t} \sbra}(v)-\log h_{\sbla \rho_{0} \sbra} (v)}t - g(u_0)
\ge g(u_t) - g(u_0) + \frac{o(t,u_t)}t,
\]
from which we can also conclude that,
\begin{equation}\label{limup}
\lim_{t \to 0^{-}}  \frac{\log h_{\sbla \rho_{t} \sbra}(v)-\log h_{\sbla \rho_{0} \sbra} (v)}t =  g(u_0).
\end{equation}
Together with \eqref{m111}, we now obtain the desired result,
\[
\lim_{t\to 0} \frac{\log h_{\sbla \rho_{t} \sbra}(v) - \log h_{\sbla \rho_{0} \sbra}(v)} t = g(u_0)=g(\alpha^*_{\sbla \rho_{0} \sbra}(v)).
\]

Now \eqref{add} holds for all
$v\in \sn\setminus \eta_{\sbla \rho_{0} \sbra}$; i.e. almost everywhere on $\sn$, with respect to spherical Lebesgue measure. Since the support functions in \eqref{add} are continuous on $\sn$, it follows that  \eqref{add} holds for all
$v\in \sn$, and gives \eqref{4-bound}.
\end{proof}

\begin{lemm}
Suppose
$\Omega \subset \sn$ is a closed set that is not contained in any closed hemisphere of $\sn$.
Suppose $\rho_0 : \Omega \to (0,\infty)$, and $g: \Omega \to \mathbb R$ are continuous.
Let $\bla \rho_{t} \bra$ be a logarithmic family of convex hulls of $(\rho_0, g)$, then,
for $q\in \mathbb R$,
\begin{equation}\label{zx4}
\lim_{t\to 0} \frac{h_{\sbla \rho_{t} \sbra}^{-q}(v) - h_{\sbla \rho_{0} \sbra}^{-q}(v)} t
=-q h_{\sbla \rho_{0} \sbra}^{-q}(v)
g(\alpha_{\sbla \rho_{0} \sbra}^{*}(v)),
\end{equation}
for all $v\in \sn\setminus \eta_{\sbla \rho_{0} \sbra}$.
Moreover, there exists
a $\delta_0 > 0$ and an $M>0$ so that
\begin{equation}\label{h-q-bound}
|h_{\sbla \rho_{t} \sbra}^{-q}(v) - h_{\sbla \rho_{0} \sbra}^{-q}(v)| \le M |t|,
\end{equation}
for all $v\in \sn$ and all $t\in (-\delta_0, \delta_0)$.
\end{lemm}

\begin{proof}
Obviously,
\begin{equation*}
\lim_{t\to 0} \frac{h_{\sbla \rho_{t} \sbra}^{-q}(v) - h_{\sbla \rho_{0} \sbra}^{-q}(v)} t
=-q h_{\sbla \rho_{0} \sbra}^{-q}(v) \lim_{t\to 0} \frac{\log h_{\sbla \rho_{t} \sbra}(v) - \log h_{\sbla \rho_{0} \sbra}(v)} t,
\end{equation*}
provided the limit on the right exists. Thus, Lemma \ref{lemma4.1} gives \eqref{zx4}.

Since $\sbla \rho_{0} \sbra$ is a convex body in $\kno$ and $\sbla \rho_{t} \sbra \to
\sbla \rho_{0} \sbra$ as $t\to 0$,  there exist $m_0, m_1 \in (0,\infty)$ and $\delta_1>0$
so that
\[
\text{$0< m_0 < h_{\sbla \rho_{t} \sbra} < m_1$,\quad on $\sn$,}
\]
for each $t\in (-\delta_1, \delta_1)$.
From this it follows that there exists $M_1 > 1$ so that
\[
\text{$0< \frac{h^{-q}_{\sbla \rho_{t} \sbra} } {h^{-q}_{\sbla \rho_{0} \sbra}}  < M_1$,\quad on $\sn$.}
\]
It is easily seen that $s-1 \ge \log s$ whenever $s\in (0,1)$ while $s-1 \le M_1 \log s$ whenever
$s \in [1,M_1]$. Thus,
\[
\text{$|s-1| \le M_1 |\log s|$, \quad when $s\in (0,M_1)$.}
\]
It follows that
\[
\Big|   \frac{h^{-q}_{\sbla \rho_{t} \sbra} } {h^{-q}_{\sbla \rho_{0} \sbra}} -1   \Big|
\le M_1 \Big|\log   \frac{h^{-q}_{\sbla \rho_{t} \sbra} } {h^{-q}_{\sbla \rho_{0} \sbra}}      \Big|,
\]
that is,
\[
\text{
$|h_{\sbla \rho_{t} \sbra}^{-q} - h_{\sbla \rho_{0} \sbra}^{-q}|
\le
|q| h_{\sbla \rho_{0} \sbra}^{-q} M_1 |\log h_{\sbla \rho_{t} \sbra }  - \log h_{\sbla \rho_{0} \sbra}|
\le
|q| \frac{M_1}{m_0^{q}} |\log h_{\sbla \rho_{t} \sbra }  - \log h_{\sbla \rho_{0} \sbra}|
$,
\ on $\sn$,}
\]
whenever $t\in (-\delta_1, \delta_1)$.
This and \eqref{4-bound} give \eqref{h-q-bound}.
\end{proof}

The derivative of radial functions of Wulff shapes is contained in the following lemma.

\begin{lemm}\label{h-W-var}
Suppose $\Omega \subset \sn$ is a closed set not contained in any closed hemisphere of $\sn$.
Let both $h_0 : \Omega \to (0,\infty)$ and $f: \Omega \to \mathbb R$ be continuous.
If $\blb h_t \brb$ is a logarithmic family of
Wulff shapes associated with $(h_0, f)$, where
\[
\log h_t(v) = \log h_0(v) + t f(v) +o(t,v),
\]
for $v\in\Omega$. Then for almost all (with respect to spherical Lebesgue measure) $u\in S^{n-1}$,
\[
\lim_{t\to 0} \frac{\log\rho_{\sblb h_t \sbrb}(u) - \log\rho_{\sblb h_0 \sbrb}(u)}t
= f(\alpha_{\sblb h_0 \sbrb}(u)).
\]
\end{lemm}

\begin{proof} Let $\rho_t = 1/h_t$.  Then
\[
\log \rho_t (v) = \log \rho_0(v) -tf(v) - o(t,v),
\]
and from Lemma \ref{lemma4.1} we know that,
\begin{equation}\label{zx-1}
\lim_{t\to 0} \frac{\log h_{\sbla \rho_{t} \sbra } (v) - \log h_{\sbla \rho_{0} \sbra}(v)} t = -f (\alpha_{\sbla \rho_{0} \sbra}^{*}(v)),
\end{equation}
for almost all (with respect to spherical Lebesgue measure) $v\in S^{n-1}$.
From Lemma \ref{wulff-hull} we have,
\begin{equation}\label{zx-2}
\blb h_t \brb = \bla \rho_t \bra^*.
\end{equation}
From  \eqref{zx-2} and \eqref{polar-identity} we have,
\begin{equation}\label{zx-3}
\log\rho_{\sblb h_t \sbrb} - \log\rho_{\sblb h_0 \sbrb} = -(\log h_{\sbla \rho_{t} \sbra} - \log h_{\sbla \rho_{0} \sbra}).
\end{equation}
But \eqref{zx-2}, when combined with Lemma \ref{measure-zero}, gives
\begin{equation}\label{zx-4}
\alpha^*_{\sbla \rho_{t} \sbra}=\alpha_{\sbla \rho_t \sbra^*}=\alpha_{\sblb h_t \sbrb},
\end{equation}
almost everywhere on $\sn$.

When \eqref{zx-1} is combined with \eqref{zx-3} and \eqref{zx-4}, we obtain the desired result.
\end{proof}

The following theorem gives the variational formula for a dual quermassintegral
in terms of its associated dual curvature measure and polar convex hull.

\begin{theo}\label{dual-volume-variation}
Suppose $\Omega \subset \sn$ is a closed set not contained in any closed hemisphere of $\sn$, and
$\rho_{0} : \Omega \to (0,\infty)$  and $g: \Omega \to \mathbb R$ are continuous. If $\bla \rho_{t} \bra$ is a logarithmic family of convex hulls of $(\rho_{0}, g)$,
then for $q\neq 0$,
\begin{align*}
\lim_{t\to 0} \frac{\wt V_q(\bla \rho_t \bra^*) -\wt V_q(\bla \rho_{0} \bra^*)}t
&= - q \int_{\Omega} g(u) \, d\wt C_q(\bla \rho_{0} \bra^*, u),  \\
\shortintertext{and}
\lim_{t\to 0} \frac{\log \bar V_0(\bla \rho_t \bra^*) -\log \bar V_0(\bla \rho_{0} \bra^*)}t
&= -\frac 1{\omega_n}\int_{\Omega} g(u) \, d\wt C_0(\bla \rho_{0} \bra^*, u),\\
\shortintertext{or equivalently, for each $q\in \mathbb R,$}
\frac{d}{dt} \log \bar V_q(\bla \rho_t \bra^*)\Big|_{t=0}
&= - \frac{1}{\wt V_q(\bla \rho_{0} \bra^*)}
\int_{\Omega} g(u) \, d\wt C_q(\bla \rho_{0} \bra^*, u).
\end{align*}
\end{theo}

\begin{proof}
Abbreviate $\eta_{\sbla \rho_{0} \sbra}$ by $\eta_0$.
Recall that $\eta_0$ is a set of
spherical Lebesgue measure zero, that
consist of the complement, in $\sn$, of the regular normal vectors of the convex body
$\bla \rho_{0} \bra=\conv\{ \rho_{0}(u)u: u\in \Omega\}$. Recall also that
the continuous function
$\alpha_{\sbla \rho_{0} \sbra}^*: \sn \setminus \eta_0 \to \sn$ is well defined by $\alpha_{\sbla \rho_{0} \sbra}^*(v) \in \balpha_{\sbla \rho_{0} \sbra}^*(v)$, for all $v\in \sn \setminus \eta_0$.

Suppose $v\in \sn \setminus \eta_0$. To see that $\{\alpha_{\sbla \rho_{0} \sbra}^*(v)\}= \balpha_{\sbla \rho_{0} \sbra}^*(v) \subset \Omega$, let
\[
h_{\sbla \rho_{0} \sbra}(v) =\max\nolimits_{u\in\Omega}\, \rho_0(u) u \cdot v = \rho_0(u_0) u_0 \cdot v,
\]
for some $u_0\in\Omega$. But this means that
\[
\rho_0(u_0) u_0 \in H_{\sbla \rho_{0} \sbra}(v),
\]
and hence $\rho_0(u_0) u_0 \in \partial\bla \rho_{0} \bra$ because it belongs to $H_{\sbla \rho_{0} \sbra}(v)$. But  $v$ is a regular normal vector of $\bla \rho_{0} \bra$, and thus $\alpha_{\sbla \rho_{0} \sbra}^*(v)=u_0\in\Omega$.
Thus,
\begin{equation}\label{zx5}
\balpha_{\sbla \rho_{0} \sbra}^*(\sn\setminus\eta_0) \subset \Omega.
\end{equation}
But \eqref{zx5} and Lemma \ref{measure-zero-bold}, now yield the fact that
\begin{equation}\label{niid}
\balpha_{\sbla \rho_{0} \sbra^*}(\sn\setminus\eta_0) \subset \Omega.
\end{equation}
Since $\Omega$ is closed, by the Tietze extension theorem
we can extend $g: \Omega \to \mathbb R$ to a continuous $\bar g: \sn \to \mathbb R$.
Therefore, using \eqref{niid} we wee that,
\begin{equation}\label{4.1-1}
g(\alpha_{\sbla \rho_{0} \sbra^*}(v)) = (\bar g  \mathbbm{1}_\Omega)(\alpha_{\sbla \rho_{0} \sbra^*}(v)),
\end{equation}
for $v\in \sn\setminus\eta_0$.

From \eqref{def-quermass-q} and \eqref{def-quermass-q-normalized-zero-v},
\eqref{polar-identity},
the fact that $\eta_0$ has measure $0$,
\eqref{h-q-bound} and the uniformly bounded convergence theorem,
\eqref{zx4} and Lemma \ref{measure-zero},
\eqref{polar-identity},
\eqref{4.1-1},
and
\eqref{cur-m-int},
we have
\begin{align*}
\lim_{t\to 0} \frac{\wt V_q(\bla \rho_{t} \bra^*) -\wt V_q(\bla \rho_{0} \bra^*)}t
&=\lim_{t\to 0}\frac1n\int_{\sn}
\frac{\rho_{\sbla \rho_{t} \sbra^*}^{q}(v) - \rho_{\sbla \rho_{0} \sbra^*}^{q}(v)} t\, dv \\
&=\lim_{t\to 0}\frac1n\int_{\sn}
\frac{h_{\sbla \rho_{t} \sbra}^{-q}(v) - h_{\sbla \rho_{0} \sbra}^{-q}(v)} t\, dv \\
&=\frac1n\int_{\sn\setminus \eta_0}  \lim_{t\to 0}
\frac{h_{\sbla \rho_{t} \sbra}^{-q}(v) - h_{\sbla \rho_{0} \sbra}^{-q}(v)} t\, dv \\
&=-\frac q n \int_{\sn\setminus \eta_0} g(\alpha^*_{\sbla \rho_{0} \sbra}(v)) h_{\sbla \rho_{0} \sbra}^{-q}(v)\, dv \\
&=-\frac q n \int_{\sn\setminus \eta_0} g(\alpha_{\sbla \rho_{0} \sbra^*}(v)) \rho_{\sbla \rho_{0} \sbra^*}^{q}(v)\, dv \\
&=-\frac q n \int_{\sn} (\bar g \mathbbm{1}_\Omega)(\alpha_{\sbla \rho_{0} \sbra^*}(v)) \rho_{\sbla \rho_{0} \sbra^*}^{q}(v)\, dv\\
&=-q \int_{\sn} (\bar g \mathbbm{1}_\Omega)(u)\, d\wt C_q(\bla \rho_{0} \bra^*, u) \\
&=-q \int_{\Omega} g(u)\, d\wt C_q(\bla \rho_{0} \bra^*, u).
\end{align*}

From  \eqref{def-quermass-q-normalized-zero}
and \eqref{def-quermass-q-normalized-zero-v},
\eqref{polar-identity},
the fact that $\eta_0$ has measure $0$ together with Lemma \ref{lemma4.1},
Lemma \ref{measure-zero},
and \eqref{4.1-1},
and
\eqref{cur-m-int},
we have
\begin{align*}
\lim_{t\to 0} \frac{\log \bar V_0(\bla \rho_{t} \bra^*) -\log \bar V_0(\bla \rho_{0} \bra^*)}t
 &=\lim_{t\to 0}\frac1{n\omega_n}\int_{\sn}
\frac{\log \rho_{\sbla \rho_{t} \sbra^*}(v) - \log \rho_{\sbla \rho_{0} \sbra^*}(v)}t \, dv \\
 &=-\lim_{t\to 0}\frac1{n\omega_n}\int_{\sn}
\frac{\log h_{\sbla \rho_{t} \sbra}(v) - \log h_{\sbla \rho_{0} \sbra}(v)}t \, dv \\
&=-\frac1{n\omega_n}\int_{\sn\setminus \eta_0} g(\alpha_{\sbla \rho_{0} \sbra}^{*}(v))\, dv \\
&=-\frac1{n\omega_n}\int_{\sn\setminus \eta_0} g(\alpha_{\sbla \rho_{0} \sbra^*}(v))\, dv \\
&=-\frac1{n\omega_n}\int_{\sn} (\bar g \mathbbm{1}_\Omega)(\alpha_{\sbla \rho_{0} \sbra^*}(v))\, dv \\
&=-\frac1{\omega_n} \int_{\sn} (\bar g \mathbbm{1}_\Omega)(u) \, d\wt C_0(\bla \rho_{0} \bra^*,u) \\
&=-\frac1{\omega_n}\int_{\Omega} g(u)\, d\wt C_0(\bla \rho_{0} \bra^*, u).
\end{align*}
\end{proof}

The following theorem gives the variational formula for a dual quermassintegral
in terms of its associated dual curvature measure and Wulff shapes.

\begin{theo}\label{dual-volume-variation-w}
Suppose $\Omega \subset \sn$ is a closed set not
contained in any closed hemisphere of $\sn$. If
$h_0 : \Omega \to (0,\infty)$ and $f: \Omega \to \rbo$ are continuous,
and $\blb h_t \brb$ is a logarithmic family of Wulff shapes associated with $(h_0, f)$,
then, for $q\neq 0$,
\begin{align*}
\lim_{t\to 0} \frac{\wt V_q(\blb h_t \brb) -\wt V_q(\blb h_0 \brb)}t
&= q \int_{\Omega} f(v) \, d\wt C_q(\blb h_0 \brb, v), \\
\shortintertext{and,}\\
\lim_{t\to 0} \frac{\log \bar V_0(\blb h_t \brb) -\log \bar V_0(\blb h_0 \brb)}t
&= \frac 1{\omega_n}\int_{\Omega} f(v) \, d\wt C_0(\blb h_0 \brb, v),
\shortintertext{or equivalently, for all $q\in \mathbb R$,}\\
\frac{d}{dt} \log \bar V_q(\blb h_t \brb)\Big|_{t=0}
&= \frac{1}{\wt V_q(\blb h_0 \brb)}
\int_{\Omega} f(v) \, d\wt C_q(\blb h_0 \brb, v).
\end{align*}
\end{theo}

\begin{proof}
The logarithmic family of Wulff shapes $\blb h_t \brb$ is defined
as the Wulff shape of $h_t$, where $h_t$ is given by
\[
\log h_t = \log h_0 + t f +o(t,\cdot).
\]
Let $\rho_t = 1/h_t$. Then
\[
\log \rho_t = \log \rho_0 - t f - o(t,\cdot).
\]
Let $\bla \rho_{t} \bra$ be the logarithmic family of convex hulls associated with $(\rho_0, -f)$.
But from Lemma \ref{wulff-hull} we know that
\[
\blb h_t \brb = \bla \rho_{t} \bra^*,
\]
and the lemma now follows from Lemma \ref{dual-volume-variation}.
\end{proof}

The variational formulas above imply variational formulas for
dual quermassintegrals of convex hull perturbations of a convex body
in terms of dual curvature measures.

Suppose $K\in \kno$ and $f: \sn \to \mathbb R$ is continuous.
We shall write $\blb K,f,t \brb$ for the
Wulff shape $\blb h_t \brb$ where $h_t$ is given by
\[
\log h_t = \log h_K + t f +o(t,\cdot).
\]
If $K\in \kno$ and $g: \sn \to \mathbb R$ is continuous,
we shall write $\bla K,g,t \bra$ for the convex hull
$\bla \rho_{t} \bra$, where $\rho_{t}$ is given by
\[
\log \rho_t = \log \rho_K + t g +o(t,\cdot).
\]

\begin{lemm}\label{d-area-var}
Suppose $K\in \kno$
and $g: \sn \to \mathbb R$ is continuous. Then, for $q\neq 0$,
\begin{align*}
\lim_{t\to 0} \frac{\wt V_q(\bla K^*, g, t \bra^*) -\wt V_q(K)}t
&= - q \int_{\sn} g(v) \, d\wt C_q(K, v), \\
\shortintertext{and}\\
\lim_{t\to 0} \frac{\log \bar V_0(\bla K^*, g, t \bra^*) -\log \bar V_0(K)}t
&= -\frac 1{\omega_n}\int_{\sn} g(v) \, d\wt C_0(K, v).\\
\shortintertext{or equivalently, for all $q\in \mathbb R$,}\\
\frac{d}{dt} \log \bar V_q(\bla K^*, g, t \bra^*)\Big|_{t=0}
&= - \frac{1}{\wt V_q(K)}
\int_{\sn} g(v) \, d\wt C_q(K, v).
\end{align*}
\end{lemm}

\begin{proof}
In Lemma \ref{dual-volume-variation}, let $\rho_0 = 1/h_K =\rho_{K^*}$. Then
$\bla \rho_{t} \bra^* = \bla K^*, g, t \bra^*$, and in particular, from \eqref{conv-of-body}
we have $\bla \rho_0 \bra^* = \bla \rho_{K^*} \bra^* = K^{**}=K$.
\end{proof}

The variational formulas of convex hulls above imply variational formulas
of Wulff shapes.

\begin{lemm}\label{dual-volume-variation-wulff}
Suppose $K\in \kno$
and $f: \sn \to \mathbb R$ is continuous. Then,  for $q\neq 0$,
\begin{align*}
\lim_{t\to 0} \frac{\wt V_q(\blb K,f,t \brb) -\wt V_q(K)}t
&= q \int_{\sn} f(v) \, d\wt C_q(K, v), \\
\shortintertext{and}\\
\lim_{t\to 0} \frac{\log \bar V_0(\blb K,f,t \brb) -\log \bar V_0(K)}t
&= \frac 1{\omega_n}\int_{\sn} f(v) \, d\wt C_0(K, v),\\
\shortintertext{or equivalently, for all $q\in \mathbb R$,}\\
\frac{d}{dt} \log \bar V_q(\blb K,f,t \brb) \Big|_{t=0}
&= \frac{1}{\wt V_q(K)}\int_{\sn} f(v) \, d\wt C_q(K, v).
\end{align*}
\end{lemm}

\begin{proof} The logarithmic family of Wulff shapes $\blb K,f,o,t \brb$ is defined
by the Wulff shape $\blb h_t \brb$ where
\[
\log h_t = \log h_K + t f +o(t,\cdot).
\]
This, and the fact that $1/h_K = \rho_{K^*}$, allows us to define
\[
\log \rho^*_t = \log \rho_{K^*} - t f - o(t,\cdot),
\]
and $\rho^*_t $ will generate a logarithmic family of convex hulls $\bla K^*,-f,-o,t \bra$.
Since $\rho^*_t = 1/h_t$, Lemma \ref{wulff-hull} gives
\[
\blb K,f,o,t \brb = \bla K^*,-f,-o,t \bra^*.
\]
The lemma now follows directly from Lemma \ref{d-area-var}.
\end{proof}

The following gives the variational formulas for dual quermassintegrals of Minkowski combinations.

\begin{coro}
Let $K\in \kno$ and $L$ be a compact convex set in $\rn$. Then, for $q\neq 0$,
\begin{align*}
\lim_{t\to 0^{+}} \frac{\wt V_q(K+tL) -\wt V_q(K)}t
&= q \int_{\sn} \frac{h_L(v)}{h_K(v)} \, d\wt C_q(K, v), \\
\shortintertext{and}\\
\lim_{t\to 0^{+}} \frac{\log \bar V_0(K+tL) -\log \bar V_0(K)}t
&= \frac 1{\omega_n}\int_{\sn} \frac{h_L(v)}{h_K(v)} \, d\wt C_0(K, v),
\shortintertext{or equivalently, for all $q\in \mathbb R$,}\\
\frac{d}{dt} \log \bar V_q(K+tL) \Big|_{t=0^{+}}
&= \frac{1}{\wt V_q(K)}\int_{\sn} \frac{h_L(v)}{h_K(v)} \, d\wt C_q(K, v).
\end{align*}
\end{coro}
\begin{proof}
From \eqref{minkowski-combination}, we have
\begin{equation}\label{ntxx}
h_{K+tL} = h_K + t h_L.
\end{equation}
From \eqref{ntxx}, it follows immediately that, for sufficiently small $t>0$,
\[
\log h_{K+tL} = \log h_K + t\frac{h_L}{h_K} +o(t,\cdot),
\]
Since $K$ and $L$ are convex, the Wulff shape $\blb h_{K+tL} \brb = K+tL$.
The desired result now follows directly from
Lemma \ref{dual-volume-variation-wulff}.
\end{proof}

The following variational formula of Aleksandrov for the volume of a convex body is a
critical ingredient in the solution of the classical Minkowski problem. The proof
given by Aleksandrov depends on the Minkowski mixed-volume inequality,
see Schneider \cite{S14}, Lemma 7.5.3.
The proof presented below is different and does not depend on inequalities for mixed volumes.

\begin{lemm}
Suppose $\Omega \subset \sn$ is a closed set not contained in any closed hemisphere of $\sn$.
Suppose $h_0 : \Omega \to (0,\infty)$ and $f: \Omega \to \mathbb R$  are continuous, and $\delta_0>0$.  If for each $t\in (-\delta_0, \delta_0)$ the function $h_t: \Omega \to (0,\infty)$ is defined by
\begin{equation}\label{onlyonce}
h_t=h_0 + t f + o(t,\cdot),
\end{equation}
where $o(t, \cdot)$ is continuous in $\Omega$ and $\lim_{t\to 0} o(t,\cdot)/t = 0$, uniformly in $\Omega$,
and
$\blb h_t \brb$ is the Wulff shape of $h_t$, then
\[
\lim_{t\to 0} \frac{V(\blb h_t\brb) - V(\blb h_0 \brb)}t = \int_{\Omega} f(v)\, dS(\blb h_0 \brb, v).
\]
\end{lemm}

\begin{proof} From \eqref{onlyonce}, it follows immediately that for sufficiently small $t$,
\[
\log h_t(v) = \log h_0(v) + t\frac{f(v)}{h_0(v)} +o(t,v).
\]
Since $\wt V_n = V$, the case $q=n$
in Lemma \ref{dual-volume-variation-w} gives
\[
\lim_{t\to 0} \frac{V(\blb h_t \brb) -V(\blb h_0 \brb)}t
= n \int_{\Omega} \frac{f(v)}{h_0(v)}  \, d\wt C_n(\blb h_0 \brb, v).
\]
But from \eqref{dual-integral} we know that, on $\Omega$,
\[
d\wt C_n(\blb h_0 \brb, \cdot) = dV_{\sblb h_0 \sbrb} = \frac1n h_{\sblb h_0 \sbrb}\, dS(\blb h_0 \brb, \cdot)
= \frac1n h_0\, dS(\blb h_0 \brb, \cdot),
\]
where the last of these follows from the well-known fact (see (7.100) in Schneider \cite{S14}) that the set of points on $\Omega$ where $h_{\sblb h_0 \sbrb}\neq h_0$ has $S(\blb h_0 \brb, \cdot)$-measure $0$.
\end{proof}

\section{Minkowski problems associated with quermassintegrals \\ and dual quermassintegrals}

Roughly speaking, Minkowski problems are characterization problems of
the differentials of geometric functionals of convex bodies.
Two families of fundamental geometric functionals of convex bodies
are quermassintegrals and dual quermassintegrals. Minkowski problems
associated with quermassintegrals have a long history and have attracted
much attention from convex geometry, differential geometry, and partial
differential equations. We first mention these Minkowski problems, and then
pose a Minkowski problem for dual quermassintegrals, called the dual Minkowski problem.

Area measures come from variations of quermassintegrals and can be viewed
as differentials of quermassintegrals. The Minkowski problem associated
with quermassintegrals is the following:

\medskip
\noindent
{\bf The Minkowski problem of area measures.}
{\it Given a finite Borel measure $\mu$ on the unit sphere $\sn$
and an integer $1\le i \le n-1$.
What are the necessary and sufficient conditions so that there
exists a convex body $K$ in $\rn$ satisfying
\[
S_i(K, \cdot) = \mu\,?
\]
}

The case of $i=n-1$ is the classical Minkowski problem.
The case of $i=1$ is the classical Christoffel problem.

\medskip

As is established in previous sections, dual curvature measures come
from variations of dual quermassintegrals and can be viewed as differentials
of dual quermassintegrals. Thus, the Minkowski problem associated with dual
quermassintegrals is dual to the Minkowski problem of area measures
which is associated with quermassintegrals.
Therefore, it is natural to pose the following dual Minkowski problem:
\smallskip

\noindent
{\bf The Minkowski problem for dual curvature measures.}
{\it Given a finite Borel measure $\mu$ on the unit sphere $\sn$ and a real number $q$.
What are necessary and sufficient conditions for the existence of
a convex body $K\in \kno$ satisfying
\[
\wt C_q(K,\cdot) = \mu\,?
\]
}

Now, \eqref{dual-integral} reminds us that the $n$-th dual curvature measure $\wt C_n(K,\cdot)$ is the
cone volume measure $V_K$. So, the case of $q=n$ of the dual Minkowski problem is
the logarithmic Minkowski problem for cone volume measure. Also \eqref{dual-integral-2} reminds us that
the $0$-th dual curvature measure $\wt C_0(K,\cdot)$ is Aleksandrov's integral curvature
of $K^*$ (with the constant factor $1/n$). Thus, the case of $q=0$ of the dual Minkowski problem is
the Aleksandrov problem. The logarithmic Minkowski problem and
the Aleksandrov problem were thought to be two entirely different problems.
It is amazing that they are now seen to be special cases of the dual Minkowski problem.

\smallskip

We will use variational method to obtain a solution to the dual Minkowski problem for the symmetric case.
The first crucial step is to associate the dual Minkowski problem to a maximization problem.
By using the variational formulas for dual quermassintegrals, we can
 transform the existence problem for the Minkowski problem for dual curvature measures to
 that of a maximization problem.

Suppose $\mu$ is a finite Borel measure on $\sn$ and $K\in\kno$.
Define
\begin{equation}
\Phi_\mu(K) = - \frac1{|\mu|} \int_{\sn} \log h_K(v) \, d\mu(v) + \log \bar V_q(K).
\end{equation}
Recall that $\bar V_q(K)$ is the normalized $(n-q)$-th dual quermassintegral of $K$.
Since $\bar V_q$ is homogeneous of degree $1$,  it follows that
$\Phi_\mu$ is homogeneous of degree 0; i.e., for $Q\in\kno$ and $\lambda >0$
\begin{equation}\label{homog5}
\Phi_\mu (\lambda Q) = \Phi_\mu(Q).
\end{equation}

The following lemma shows that a solution to the dual Minkowski problem for the measure $\mu$
is also a solution to a maximization problem for the functional $\Phi_\mu$.

\begin{lemm}\label{lem-max2}
Suppose $q\in\rbo$ and $\mu$ is a finite even Borel measure on $\sn$
with $|\mu| >0$, but when $q=0$, we are given that $|\mu|=\omega_n$. If $K\in\kne$ with
$\wt V_q(K)=|\mu|$ is such that
\[
\sup\{\Phi_\mu(Q) : \text{$\wt V_q(Q)=|\mu|$ and $Q \in \kne$} \} = \Phi_\mu(K),
\]
then
\[
\wt C_q(K,\cdot) = \mu.
\]
\end{lemm}

\begin{proof}
On $C_e^+(\sn)$, the class of strictly positive continuous even functions on $\sn$, define the functional
$\Phi:C_e^+(\sn) \to \rbo$, by letting for each $f\in C_e^+(\sn)$,
\begin{equation}\label{phi}
\Phi(f)= \frac1{|\mu|}\int_{\sn} \log f \, d\mu  + \log \bar V_q(\bla f \bra^*).
\end{equation}
Note that since $f$ is even, it follows that $\bla f \bra = \conv \{f(u)u : u\in\sn \}\in\kne$. We first observe that $\Phi$ is homogeneous of degree $0$, in that for all $\lambda>0$, and all $f\in C_e^+(\sn)$,
\[
\Phi(\lambda f)= \Phi(f).
\]
To see this, first recall the fact that, from its definition, $\bar V_q$ is obviously homogeneous of degree $1$, while clearly $\bla \lambda f \bra = \lambda \bla  f \bra$ and thus $\bla \lambda f \bra^* = \lambda^{-1} \bla  f \bra^*$.

To see that $\Phi:C_e^+(\sn) \to \rbo$ is continuous, recall that
if $f_0, f_1,\ldots \in C_e^+(\sn)$, are such that
$\lim_{k\to\infty}f_k = f_0$, uniformly on $\sn$, then
$\bla f_k \bra \to \bla f_0 \bra$,
and thus $\bla f_k \bra^* \to \bla f_0 \bra^*$.
Since $\bar V_q:\kno \to (0,\infty)$ is continuous, the continuity of $\Phi$ follows.

Consider the maximization problem
\begin{equation}\label{max-prob}
\sup\{\Phi(f) : f\in C_e^+(\sn) \}.
\end{equation}

For the convex hull ${\bla f \bra}=\conv\{ f(u)u : u\in\sn  \}$, of
$f\in C_e^+(\sn)$,
we clearly have  $\rho_{{\sbla f \sbra}} \ge f$
and also $\bla {\rho_{{\sbla f \sbra}}} \bra = {\bla f \bra}$, from \eqref{conv-of-body}, and thus
$\bla {\rho_{{\sbla f \sbra}}} \bra^* = {\bla f \bra}^*$.
Thus, directly from \eqref{phi}, we have
\[
\Phi(f) \le \Phi(\rho_{{\sbla f \sbra}}).
\]
This tells us that in searching for the supremum in \eqref{max-prob} we can restrict our attention to the radial functions of bodies in $\kne$; i.e.,
\[
\sup\{\Phi(f) : f\in C_e^+(\sn) \}=\sup\{\Phi(\rho_Q) : Q\in\kne \}.
\]

Therefore, a convex body $K_0\in \kne$ satisfies
\[
\Phi_\mu(K_0^*) = \sup\{\Phi_\mu(Q^*) : Q^*\in \kne\},
\]
if and only if
\[
\Phi(\rho_{K_0})=\sup\{\Phi(f) : f\in C_e^+(\sn) \}.
\]
From \eqref{homog5}
we see that we can always restrict our search to bodies $Q\in\kne$ for which $\wt V_q(Q^*)=|\mu|$, when $q\neq0$.
When
$q=0$, in order that
$\wt V_0(Q^*) = |\mu|$, requires that $|\mu|=\omega_n$, since $\wt V_0(Q)=\omega_n$, for all bodies $Q$.
Thus, we can restrict our attention to those bodies such that $\wt V_q(Q^*)=|\mu|$.

Suppose $K_0^*\in\kne$ is a maximizer for $\Phi_\mu$, or equivalently $\rho_{K_0}$ is a maximizer for $\Phi$; i.e.,
\[
\Phi_\mu(K_0^*) = \sup\{\Phi_\mu(Q^*) : \text{$\wt V_q(Q^*)=|\mu|$ and  $Q^*\in \kne$}\}.
\]
Choose $g : \sn \to \rbo$ a continuous arbitrary but fixed even function.
For $\delta>0$ and $t\in (-\delta, \delta)$, define $\rho_t$ by
\[
\rho_t= \rho_{K_0} e^{t g},
\]
or equivalently,
\begin{equation}\label{rhot}
\log \rho_t = \log\rho_{K_0} + tg,
\end{equation}
Let $\bla \rho_t \bra =\bla K_0, g,t \bra$ be the logarithmic family of convex hulls associated with $(K_0, g)$.
Since $\bla \rho_0 \bra  = \bla  \rho_{K_0} \bra = K_0$, we can use
Lemma \ref{dual-volume-variation} to conclude that,
\begin{equation}\label{ttt}
\frac{d}{dt} \log \bar V_q(\bla K_0, g,t \bra^*)\Big|_{t=0} = - \frac{1}{\wt V_q(K_0^*)}
\int_{\sn} g(u) \, d\wt C_q(K_0^*, u).
\end{equation}
From the fact that $\rho_0=\rho_{K_0}$ is a maximizer for $\Phi$, and definition \eqref{phi}
we have
\[
0=\frac{d}{dt} \Phi(\rho_t)\Big|_{t=0} = \frac{d}{dt}\left( \frac1{|\mu|}
\int_{\sn} \log \rho_t(u) \, d\mu(u) + \log \bar V_q(\bla K_0,g,t\bra^*)
\right)\Big|_{t=0}.
\]
This, together with \eqref{rhot} and \eqref{ttt}, shows that
\begin{equation}\label{r3r}
\frac1{|\mu|} \int_{\sn} g(u)\, d\mu(u)
- \frac1{\wt V_q(K_0^*)}\int_{\sn} g(u)\, d\wt C_q(K_0^*, u)=0.
\end{equation}
Since $\wt V_q(K_0^*) = |\mu|$, and since \eqref{r3r} must hold for all
continuous even $g : \sn \to \rbo$,
we conclude that $\mu = \wt C_q(K_0^*, \cdot)$.
\end{proof}

Since this paper aims at a solution to the dual Minkowski problem for origin-symmetric convex bodies,
Lemma \ref{lem-max2} is stated and proved only for even measures and origin-symmetric
convex bodies. However, similar result holds for general measures and convex bodies that
contain the origin in their interiors. The above proof works, mutatis mutandis.

\section{Solving the maximization problem \\ associated with the dual Minkowski problem }

In the previous section, by using a variational argument,
we showed that the existence of a solution to a certain maximization problem
would imply the existence of a solution to the dual Minkowski problem.
In this section, we show that the maximization problem does indeed have a solution. The key
is to prove compactness and non-degeneracy, that is, the convergence of a maximizing
sequence of convex bodies to a convex body (a compact convex set with non-empty interior).
This requires delicate estimates of
dual quermassintegrals of polytopes and entropy-type integrals with
respect to the given measure in the dual Minkowski problem.

Throughout this section, for real  $p>0$, we shall use $p'$ to denote the H\"older conjugate of $p$.
Also, the expression $c_1=c(n,k,N)$ will be used to mean that $c_1$ is a \lq\lq constant\rq\rq depending on only the values of $n$, $k$, and $N$.

\subsection{Dual quermassintegrals of cross polytopes}
Let $e_1, \dots, e_n$ be orthogonal unit vectors
and $a_1, \ldots, a_n\in (0,\infty)$. The convex body
\[
P=\{x\in \rn : \text{$|x\cdot e_i|\le a_i$ for all $i$}\}
\]
is a rectangular parallelotope centered at the origin. The parallelotope $P$ is the Minkowski sum of the line segments whose support functions are $x \mapsto a_i |x\cdot e_i|$, and hence the support function of $P$ is given by
\[
h_P(x) = \sum_{i=1}^n a_i |x\cdot e_i|,
\]
for $x\in\rn$.
The polar body $P^*$ is a cross polytope. From \eqref{polar-identity} we know that the radial function
of $P^*$ is given by
\[
\rho_{P^*}(x) = 1/h_P(x)= \left(\sum_{i=1}^n a_i |x\cdot e_i|\right)^{-1},
\]
for $x\in\rn\setminus \{0\}$.
From \eqref{def-quermass-q} we know that for the $(n-q)$-th dual quermassintegral of the cross polytope $P^*$  we have,
\begin{equation}\label{s1-1}
\wt W_{n-q}(P^*) = \frac1n \int_{\sn} \left(\sum_{i=1}^n a_i |u\cdot e_i|\right)^{-q} du.
\end{equation}

From \eqref{def-quermass-q} we see that
when $q=n$,  
\eqref{s1-1} becomes the (well-known) volume of a cross-polytope:
\begin{equation}\label{rst}
V(P^*) = \frac{2^n}{n!} (a_1 \cdots a_n)^{-1},
\end{equation}
and thus,
\begin{equation}\label{6.1-1}
\int_{\sn} \left(\sum_{i=1}^n a_i |u\cdot e_i|\right)^{-n} \, du =   \frac{2^n}{(n-1)!}(a_1 \cdots a_n)^{-1}.
\end{equation}

When some of the $a_i$ are small, the dual quermassintegral
$\wt W_{n-q}(P^*)$ becomes large. The following lemma gives a critical estimate for the size
of the dual quermassintegral.

\begin{lemm}\label{dmp-lem1}
Suppose $q\in (0,n] $, and $k$ is an integer such that $1\le k < n$.
Suppose also that
$e_1, \dots, e_n$ is an orthonormal basis in $\rn$, and 
$\varepsilon_0 >0$.
If $a_1, \ldots, a_n\in(0,\infty)$ with
$a_{k+1}, \ldots, a_n \in (\varepsilon_0,\infty)$,
then
\begin{equation}\label{d-mp-e1}
\frac1q \log \int_{\sn} \left(\sum_{i=1}^n a_i |u\cdot e_i|\right)^{-q} \, du
\le -\frac1N\log(a_1\cdots a_k) + c_0,
\end{equation}
where
\begin{equation*}
N= \begin{cases}
n             &\text{when $q=n$,}\\
\infty        &\text{when $0<q<1$,}\\
\theta       &\text{when $1\le q<n$,}
\end{cases}
\end{equation*}
where $\theta$ can be chosen to be any real number such that $\frac{q-1}{(n-1)q} <\frac1\theta < \frac1n$,
and
$c_0>0$ is
\begin{equation*}
c_0= \begin{cases}
c(k, n, \varepsilon_0)                                               &\text{when $q=n$,}\\
c(q, k, n, \varepsilon_0, N)           &\text{when $1\le q<n$,}\\
c(q, k, n, \varepsilon_0)                &\text{when $0<q<1$ or when $q=n$.}
\end{cases}
\end{equation*}
\end{lemm}

\begin{proof}
When $q=n$, \eqref{d-mp-e1} follows directly from \eqref{6.1-1}.

Now consider the case where $0<q<n$.
Write $\rn = \mathbb R^{k} \times \mathbb R^{n-k}$,  with $\{e_1, \dots, e_k\} \subset \mathbb R^k$,
 and $\{e_{k+1}, \dots, e_n\} \subset \mathbb R^{n-k}$. Consider the general spherical coordinates
\[
u=(u_2\cos\varphi, u_1\sin\varphi), \quad u_2 \in  S^{k-1} \subset \mathbb R^{k},\
 u_1 \in  S^{n-k-1} \subset \mathbb R^{n-k}, \quad 0\le \varphi \le \frac\pi2.
\]
For spherical Lebesgue measures on $\sn$ and its subspheres we have
(see \cite{GZ99}),
\begin{equation}\label{spherical-u}
du = \cos^{k-1}\varphi \, \sin^{n-k-1}\varphi \,  d\varphi du_2 du_1.
\end{equation}

Let
\[
h_1(u_1) = \sum_{i=k+1}^n a_i |u_1\cdot e_i|, \quad h_2(u_2) = \sum_{i=1}^k a_i |u_2\cdot e_i|,
\]
be the support functions of the corresponding rectangular parallelotopes in $\mathbb R^{n-k}$ and $\mathbb R^{k}$.
Throughout $N$ will be chosen so that $N\ge n$. Let $p=N/k>1$.
From \eqref{spherical-u} and Young's inequality, we have
\begin{align}
\int_{\sn} &\left(\sum_{i=1}^n a_i |u\cdot e_i|\right)^{-q} \, du
= \int_{\sn} \left(h_2(u_2) \cos\varphi + h_1(u_1) \sin\varphi \right)^{-q}\, du \nonumber\\
&\le \int_0^\frac\pi2 \int_{S^{n-k-1}} \int_{S^{k-1}}
(p'h_1(u_1)\sin\varphi)^{-\frac q{p'}} (ph_2(u_2)\cos\varphi)^{-\frac q{p}}
(\sin\varphi)^{n-k-1} (\cos\varphi)^{k-1} \, d\varphi du_1 du_2 \nonumber\\
&= c_2
\int_{S^{n-k-1}} h_1(u_1)^{-\frac q{p'}} \, du_1
 \int_{S^{k-1}} h_2(u_2)^{-\frac q{p}} \, du_2,  \label{s1-0}
\end{align}
where
\[
c_2= ({p'}^{\frac 1{p'}}  {p}^{\frac 1{p}})^{-q} \int_0^\frac\pi2
(\sin\varphi)^{n-k-1-\frac q{p'}} (\cos\varphi)^{k-1-\frac qp} \, d\varphi=
\frac12 ({p'}^{\frac 1{p'}}  {p}^{\frac 1{p}})^{-q}  \Beta \left({\frac{n-k-\frac q{p'}}2, \frac{k-\frac qp}2}\right).
\]
The integral above is on $(0,\infty)$ provided both
\begin{equation}\label{s1-1-1}
k-\frac qp>0\quad \text{and}\quad n-k-\frac q{p'}>0,
\end{equation}
which, since $N=pk$, can be written as
\begin{equation}\label{s1-2}
\quad \frac1N < \frac1q, \quad \text{and}\quad \frac1N > \frac1k \left(1-\frac nq\right)+\frac1q.
\end{equation}
From $1\le k < n$ and $0<q<n$, we know that
\[
\frac1q > \frac1n, \quad \text{and}\quad \frac1k \left(1-\frac nq\right)+\frac1q \le \frac1{n-1}\left(1-\frac nq\right)+\frac1q
=\frac1{(n-1)q'}.
\]
Thus, the inequalities in \eqref{s1-2} and thus the inequalities in \eqref{s1-1-1} will be satisfied whenever $N$ can be chosen so that,
\begin{equation}\label{s1-3}
\frac1{(n-1)q'} <\frac1N <\frac1n.
\end{equation}
Since we are dealing with the case where $0<q<n$, such a choice of $N$ is always possible.
Note that all of this continues to hold in the cases where $N=\infty=p$ and where $p=1$, 
mutatis mutandis.

Consider first the subcase where $1\le q<n$.
Jensen's inequality, together with the left inequality in \eqref{s1-1-1},
and \eqref{6.1-1} in $\mathbb R^k$, gives
\begin{equation}\label{s1-4}
\left(\frac1{k\omega_{k}}\int_{S^{k-1}} h_2(u_2)^{-\frac q{p}}\, du_2\right)^{\frac pq}
\le
\left(\frac1{k\omega_{k}}\int_{S^{k-1}} h_2(u_2)^{-k}\, du_2\right)^{\frac1{k}}
=c_3 (a_{1} \cdots a_k)^{-\frac1k},
\end{equation}
Jensen's inequality, together with the right inequality in \eqref{s1-1-1}, and \eqref{6.1-1} in $\mathbb R^{n-k}$, when combined with the fact that $a_{k+1}, \ldots, a_n > \varepsilon_0$, gives
\begin{equation}\label{dmp-j1}
\left(\int_{S^{n-k-1}}  \negthinspace\negthinspace\negthinspace \negthinspace\negthinspace\negthinspace  h_1(u_1)^{-\frac q{p'}}\, du_1\right)^{\frac{p'}q}
\le c_4\left(   \int_{S^{n-k-1}} \negthinspace\negthinspace\negthinspace \negthinspace\negthinspace\negthinspace h_1(u_1)^{k-n}\, du_1\right)^{\frac1{n-k}}
=c_5 (a_{k+1} \cdots a_n)^{\frac1{k-n}}
\le c_6,
\end{equation}
where $c_3, \ldots, c_6$ are $c(q,k,n,N,\varepsilon_0)$ constants.

Since $p>1$, we know $p'$ is positive.
Using \eqref{s1-0}, \eqref{dmp-j1}, and \eqref{s1-4}, we get
\[
\log \int_{\sn} \left(\sum_{i=1}^n a_i |u\cdot e_i|\right)^{-q} \, du
\le
\log c_2 + \frac{q}{p'} \log c_6 +\frac{q}p \log c_3 + \log(k\omega_k)
-\frac{q}{pk} \log(a_1\cdots a_k).
\]
Since $N=pk$, this gives \eqref{d-mp-e1} for the case where $1\le q<n$.

Finally, we treat the subcase where $0<q<1$.
In this case, the H\" older conjugate  $q'<0$, and
to satisfy \eqref{s1-3} we may take $N$ to be arbitrary large.
Taking the limit $p \to \infty$ (and hence $p' \to 1$) turns
\eqref{s1-0} into
\begin{equation}\label{oops}
\int_{\sn} \left(\sum_{i=1}^n a_i |u\cdot e_i|\right)^{-q} \, du
\le c'_2
\int_{S^{n-k-1}} h_1(u_1)^{-q} \, du_1
 \int_{S^{k-1}}  \, du_2,
\end{equation}
where
\[
c'_2=
\frac12  \Beta \left({\frac{n-k-q}2, \frac{k}2}\right).
\]
which is positive and depends only on $q, k, n$.
In this subcase, inequalities \eqref{dmp-j1} become
\begin{equation}\label{s1-5}
\left(\int_{S^{n-k-1}}  \negthinspace\negthinspace\negthinspace \negthinspace\negthinspace\negthinspace  h_1(u_1)^{-q}\, du_1\right)^{\frac{1}q}
\le c'_4\left(   \int_{S^{n-k-1}} \negthinspace\negthinspace\negthinspace \negthinspace\negthinspace\negthinspace h_1(u_1)^{k-n}\, du_1\right)^{\frac1{n-k}}
=c'_5 (a_{k+1} \cdots a_n)^{\frac1{k-n}}
\le c'_6,
\end{equation}

Thus, using  \eqref{oops} and \eqref{s1-5}, gives
\[
\frac1q \log \int_{\sn} \left(\sum_{i=1}^n a_i |u\cdot e_i|\right)^{-q} \, du
\le \frac1q \log c'_2 + \log c'_6 +\frac1q\log (k\omega_k),
\]
which gives \eqref{d-mp-e1} in the subcase where $0< q<1$.
\end{proof}

\subsection{An elementary entropy-type inequality}\label{6.2}

As a technical tool, the following elementary entropy-type inequality
is needed.

\begin{lemm}\label{dmp-lem2}
Suppose $N\in(0,\infty)$ and $\alpha_1, \ldots, \alpha_n\in (0,\infty)$ are such that
\begin{equation}\label{sick}
\alpha_i + \cdots + \alpha_n
< 1-(i-1)/N,\ \text{for all $i>1$},\quad\text{while}\quad \alpha_1 + \cdots + \alpha_n=1.
\end{equation}
Then there exists a small $t>0$,
such that
\[
\sum_{i=1}^n \alpha_i \log a_i \le \frac{1+t}N \log (a_1\cdots a_n)
+ \left (1-\frac{n(1+t)}N\right) \log a_n.
\]
for all $a_1, \ldots, a_n\in(0,\infty)$, such that
$a_1 \le a_2 \le \cdots \le a_n$.
\end{lemm}

Note that $t=t(\alpha_1, \ldots, \alpha_n, N)$ is independent of any of the $a_1, \ldots, a_n\in(0,\infty)$.

\begin{proof}
Let $t>0$ be sufficiently small such that, for all $i>1$,
\begin{equation}\label{st12}
\alpha_i + \cdots + \alpha_n < 1-\frac{i-1}N (1+t) = 1- (i-1)\lambda,
\end{equation}
where $\lambda = (1+t)/N$. Let,
\[
\text{$\beta_i=\alpha_i -\lambda$, for $i<n$, \quad while\quad
$\beta_n = \alpha_n +(n-1)\lambda -1$,}
\]
and let
\[
\text{$s_i=\beta_i + \cdots +\beta_n$,\quad for $i=1, \ldots, n$,\quad\quad and\quad $s_{n+1}=0$.}
\]
Then, not only is $s_{n+1}=0$, but also
\[
s_1 = \beta_1+\cdots+\beta_n = \alpha_1+\cdots+\alpha_n -(n-1)\lambda + (n-1)\lambda -1=0,
\]
while for $i>1$,
\[
s_i = \beta_i+\cdots+\beta_n = \alpha_i+\cdots+\alpha_n -(n-i+1)\lambda +n\lambda -1
=\alpha_i+\cdots+\alpha_n +(i-1)\lambda -1 <0,
\]
by \eqref{st12}.
Now,
\begin{align*}
\sum_{i=1}^n \beta_i  \log a_i &=\sum_{i=1}^n (s_i-s_{i+1})  \log a_i \\
&=\sum_{i=1}^n s_i  \log a_i - \sum_{i=1}^n s_{i+1}  \log a_i \\
&=\sum_{i=1}^{n-1} s_{i+1}  \log a_{i+1} +s_1 \log a_1 -  \sum_{i=1}^{n-1} s_{i+1}
\log a_i -s_{n+1} \log a_n \\
&=\sum_{i=1}^{n-1} s_{i+1}  (\log a_{i+1} - \log a_i) \\
&\le 0,
\end{align*}
since the $a_i$ are monotone non-decreasing. Therefore,
\begin{align*}
\sum_{i=1}^n \alpha_i  \log a_i &=
\sum_{i=1}^n (\beta_i+\lambda)  \log a_i +(1-n\lambda) \log a_n \\
&=\sum_{i=1}^n \beta_i \log a_i + \lambda \log(a_1\cdots a_n) +(1-n\lambda) \log a_n \\
&\le  \frac{1+t}N \log (a_1\cdots a_n) + \left (1-\frac{n(1+t)}N\right) \log a_n.
\end{align*}
\end{proof}

\subsection{Estimation of an entropy-type integral with respect to a measure}

We first define a partition of  the unit sphere. Then we use the partition to estimate an entropy type
integral by an entropy-type finite sum treated in section \ref{6.2}.

Let $e_1, \dots, e_n$ be a fixed orthonormal basis for $\rn$.
Relative to this basis, for each $i=1,\ldots,n$,
define $S^{n-i }=\sn \cap \spane\{e_{i}, \dots, e_{n}\}$. For convenience
define $S^{-1} = \varnothing$.

For small $\delta \in ( 0, \frac1{\sqrt{n}} )$, define a partition of $\sn$, with respect to the orthonormal basis $e_1, \dots, e_n$, by letting
\begin{equation}\label{partition}
\Omega_{i,\delta}=\{u\in \sn : \text{$|u\cdot e_i|\ge \delta$, and $|u\cdot e_j|<\delta$ for $j<i$}\},
\quad i=1, 2, \ldots, n.
\end{equation}
Explicitly,
\begin{align*}
\Omega_{1,\delta}&=\{u\in \sn : |u\cdot e_1|\ge \delta \} \\
\Omega_{2,\delta}&=\{u\in \sn : |u\cdot e_2|\ge \delta, \ |u\cdot e_1|<\delta \} \\
\Omega_{3,\delta}&=\{u\in \sn : |u\cdot e_3|\ge \delta,
\ |u\cdot e_1|<\delta, \ |u\cdot e_2|<\delta \} \\
&\ \, \vdots \\
\Omega_{n,\delta}&=\{u\in \sn : |u\cdot e_n|\ge \delta, \ |u\cdot e_1|<\delta, \, \dots, \,
|u\cdot u_{n-1}|<\delta \}.
\end{align*}
These sets are non-empty since $e_i \in \Omega_{i,\delta}$. They are obviously disjoint.
For $\delta \in (0,\frac1{\sqrt n})$ and each $u\in \sn$, there is an $e_i$ such that
$|u\cdot e_i| \ge \delta$ and for the smallest such $i$, say $i_0$, we'll have $u\in \Omega_{{i_0},\delta}$. Thus, the union of $\Omega_{i,\delta}$ covers $\sn$.

If we let
\begin{align*}
\Omega'_{i,\delta}
&=
\{u\in \sn : \text{$|u\cdot e_i|\ge \delta$, and $|u\cdot e_j|=0$ for $j<i$}\} \\
\shortintertext{and}
\Omega''_{i,\delta}
&=
\{u\in \sn : \text{$|u\cdot e_i|>0$, and $|u\cdot e_j|<\delta$ for $j<i$}\},
\end{align*}
then
\begin{equation}\label{parot}
\Omega'_{i,\delta} \subset \Omega_{i,\delta} \subset \Omega''_{i,\delta}.
\end{equation}
As $\delta$ decreases to 0, the set $\Omega'_{i,\delta}$ increases (with respect to set inclusion)
to $S^{n-i} \setminus S^{n-i-1}$, while
 $\Omega''_{i,\delta}$ decreases
to $S^{n-i} \setminus S^{n-i-1}$.

Suppose now that $\mu$ is
a finite Borel measure
on $\sn$. From the definitions $\Omega'_{i,\delta}$ and $\Omega''_{i,\delta}$,
we conclude that
\begin{align*}
\lim_{\delta \to 0^+} \mu(\Omega'_{i,\delta})
= \mu(S^{n-i} \setminus S^{n-i-1}) \\
\shortintertext{and also}
\lim_{\delta \to 0^+} \mu(\Omega''_{i,\delta})
= \mu(S^{n-i} \setminus S^{n-i-1}).
\end{align*}
This, together with \eqref{parot}, gives
\begin{equation}\label{partition1}
\lim_{\delta \to 0^+} \mu(\Omega_{i,\delta})
= \mu(S^{n-i} \setminus S^{n-i-1}).
\end{equation}
It follows that, for each integer $1\le k \le n$,
\begin{equation}\label{partition2}
\lim_{\delta \to 0^+} \mu\Big(\bigcup_{i=k}^n\Omega_{i,\delta}\Big)
=\lim_{\delta \to 0^+}\sum_{i=k}^n \mu(\Omega_{i,\delta})
=\sum_{i=k}^n \mu(S^{n-i} \setminus S^{n-i-1})
= \mu(S^{n-k}),
\end{equation}
where $S^{n-k}=\sn \cap \spane\{e_{k}, \dots, e_{n}\}$.

\begin{lemm}\label{dmp-lem3}
Let $\mu$ be a finite Borel measure on $\sn$.  For $l=1,2,\dots$, let $a_{1l}, \dots, a_{nl}$ be $n$ sequences
in $(0,\infty)$, and let $e_{1l}, \dots, e_{nl}$ be a sequence  of orthonormal bases in $\rn$, that
converges to the orthonormal basis $e_1, \dots, e_n$. For each $i=1,\ldots,n$, and small $\delta \in ( 0, \frac1{\sqrt{n}} )$, let
\[
\Omega_{i,\delta}=\{u\in \sn : \text{$|u\cdot e_i|\ge \delta$, and $|u\cdot e_j|<\delta$ for $j<i$}\}.
\]
Then for a given small $\delta>0$, there exists an integer $L$ such that, for all $l>L$,
\[
\frac{1}{|\mu|} \int_{\sn} \log \sum_{i=1}^n \frac{|u\cdot e_{il}|}{a_{il}}\, d\mu(u)\
\ge\
\log\frac\delta2 -
\sum_{i=1}^n \frac{\mu(\Omega_{i,\delta})}{|\mu|} \log a_{il} .
\]
\end{lemm}

\begin{proof}
Since $e_{1l}, \dots, e_{nl}$ converge to
$e_1, \dots, e_n$,
for a given small $\delta>0$, there exists an $L$ such that for all $l>L$,
\[
|e_{il}-e_i|<\frac\delta2,
\]
for all $i$.
Then for $l>L$, for $u\in \Omega_{i,\delta}$,
\begin{equation}\label{s3.1}
|u\cdot e_{il}| \ge |u\cdot e_i| - |u\cdot (e_{il}-e_i)|
\ge |u\cdot e_i| - |e_{il}-e_i|
\ge \frac\delta2,
\end{equation}
for all $i$.
Therefore, for $l>L$, by using the partition $\sn = \cup_{i=1}^n \Omega_{i,\delta}$,
together with \eqref{s3.1}, we have
\begin{align*}
\int_{\sn} \log \sum_{j=1}^n \frac{|u\cdot e_{jl}|}{a_{jl}}\, d\mu(u)
&=
\sum_{i=1}^n \int_{\Omega_{i,\delta}} \log \sum_{j=1}^n \frac{|u\cdot e_{jl}|}{a_{jl}}\, d\mu(u) \\
&\ge
 \sum_{i=1}^n \int_{\Omega_{i,\delta}} \log \frac{|u\cdot e_{il}|}{a_{il}}\, d\mu(u) \\
&=
\sum_{i=1}^n \int_{\Omega_{i,\delta}} \log |u\cdot e_{il}| \, d\mu
-\sum_{i=1}^n \mu(\Omega_{i,\delta}) \log a_{il}\\
&\ge
 \mu(\sn) \log\frac\delta2
- \sum_{i=1}^n \mu(\Omega_{i,\delta}) \log a_{il}.
\end{align*}
\end{proof}

\subsection{The subspace mass inequality and non-degeneracy}

Let $\mu$ be a non-zero finite Borel measure on $\sn$.
Fix an ordered orthonormal basis
$\beta: e_1, \dots, e_n$ in $\rn$. Let $\xi_{n-i+1} =\spane\{e_i, \dots, e_n\}$
be the subspace spanned by $e_i, \dots, e_n$.  

Suppose $q\in[1,n]$.
We will say that $\mu$ satisfies the {\it $q$-th subspace mass inequality with respect to
the basis $\beta$}, if
\begin{equation}\label{smi}
\frac{\mu({\sn}\cap \xi_{n-i})}{|\mu|} < 1-\frac{i}{(n-1)q'},
\end{equation}
for all $i<n$. We will say that $\mu$ satisfies the {\it $q$-th subspace mass inequality} if it does so with respect to
every orthonormal basis.

For $q\in(0,1)$, we will say that $\mu$ satisfies the $q$-th subspace mass inequality, with respect to
the basis $\beta$, if
\begin{equation}\label{smi-1}
\frac{\mu({\sn}\cap \xi_{n-1})}{|\mu|} < 1,
\end{equation}
and if this is the case for every orthonormal basis, we shall say that $\mu$ satisfies the {\it $q$-th subspace mass inequality}.

When $q=1$, obviously $q'=\infty$, and thus the measure $\mu$ satisfies the  $q$-th subspace mass inequality \eqref{smi} if
\begin{equation*}
\frac{\mu({\sn}\cap \xi_{n-i})}{|\mu|} < 1,
\end{equation*}
for all $i<n$, for each basis $\beta$. Here, obviously the case $i=1$ implies all $i<n$. Observe that this is equivalent to the definition for $q\in (0,1)$.

When $q=n$, the subspace mass inequality \eqref{smi}  is also called
the {\it strict subspace concentration} condition, see \cite{BLYZ13jams},
\[
\frac{\mu({\sn}\cap \xi)}{|\mu|} < \frac{\dim(\xi)}{n},
\]
for each subspace $\xi$.

The following lemma will be used to show that the limit of a maximizing sequence,
 for the maximization problem associated with the dual Minkowski problem, will not be
 a degenerate compact convex set provided that the given measure satisfies the subspace mass inequality.

\begin{lemm}\label{dmp-lem4}
Suppose $\mu$ is a non-zero finite Borel measure on $\sn$ and  $q\in(0,n]$.
Suppose
$a_{1l}, \ldots, a_{nl}$ are $n$ sequences in $(0,\infty)$,
for which there exists an $\varepsilon_0>0$ and $M_0\in (0,\infty)$
such that, for all $l$,
\[
a_{1l}\le a_{2l}\le \dots \le a_{nl}\le M_0,
\]
and for some integer, $1\le k<n$,
\[
a_{1l}, \dots, a_{kl} \to 0\ \text{as $l\to\infty$,\quad while}\quad
a_{k+1,l}, \ldots, a_{n,l} > \varepsilon_0.
\]
Let $e_{1l}, \dots, e_{nl}$
be a sequence of orthonormal bases in $\rn$ that
converges to an orthonormal basis $e_1, \dots, e_n$.
If $\mu$ satisfies the $q$-th subspace mass inequality, with respect to the
orthonormal basis $e_1, \dots, e_n$, then,
\begin{equation}
-\frac{1}{|\mu|} \int_{\sn} \log\, \sum_{i=1}^n \frac{|u\cdot e_{il}|}{a_{il}}\, d\mu(u)\ + \
\frac1q \log \int_{\sn} \left(\sum_{i=1}^n a_{il} |u\cdot e_{il}|\right)^{-q} du\
\to\  -\infty,
\end{equation}
as $l\to \infty$.
\end{lemm}

\begin{proof} First, consider the case $q \in [1,n]$.

When $q \in [1,n)$, we use the fact that $\mu$ satisfies the $q$-th subspace mass inequality, with respect to the orthonormal basis $e_1, \dots, e_n$, to deduce the existence of an $N\in (n,(n-1)q')$ such that, for all $i>1$,
\begin{equation}\label{s4.1}
\frac{\mu({\sn}\cap\xi_{n-i+1})}{|\mu|}
< 1-\frac{i-1}N \le 1-\frac{i-1}{(n-1)q'},
\end{equation}
where $\xi_{n-i+1}=\spane\{e_i, \dots, e_n\}$. When $q=n$, let $N=n$, and note that \eqref{s4.1} holds.

For a small $\delta>0$, and each $i=1, \dots, n$, let $\Omega_{i,\delta}$ be the partition
defined in \eqref{partition}, with respect to the orthonormal basis $e_1, \dots, e_n$, and
let
\[
\alpha_i = \alpha_i(\delta) = \frac{\mu(\Omega_{i,\delta})}{|\mu|},
\]
for each $i$. From \eqref{partition2}, we see that, as $\delta \to 0^+$,
\[
\alpha_i + \cdots +\alpha_n \to \frac{\mu({\sn}\cap\xi_{n-i+1})}{|\mu|}\]
for each $i$.  This and \eqref{s4.1} tells us that
we can choose $\delta>0$ sufficiently small so that
\[
\alpha_i + \cdots +\alpha_n <  1-\frac{i-1}N,
\]
for each $i>1$. Note that the $\alpha_i$ satisfy the conditions of \eqref{sick} of Lemma \ref{dmp-lem2}.

The fact that $\alpha_i = {\mu(\Omega_{i,\delta})}/{|\mu|}$, when combined with
Lemma \ref{dmp-lem3}, followed by the fact that $a_{k+1,l}, \ldots, a_{nl} > \varepsilon_0$,
together with Lemma \ref{dmp-lem1}, and followed by
given that
$a_{1l}\le a_{2l}\le \dots \le a_{nl}$,
Lemma \ref{dmp-lem2}, yields the existence of a $t>0$, such that
\begin{align}\label{z0z}
-\frac{1}{|\mu|} \int_{\sn} &\log \,\sum_{i=1}^n \frac{|u\cdot e_{il}|}{a_{il}} \,d\mu(u)
+\frac1q \log \int_{\sn} \left(\sum_{i=1}^n a_{il} |u\cdot e_{il}|\right)^{-q} \, du\nonumber \\
&\le \sum_{i=1}^n \alpha_i \log a_{il} -\log \frac\delta2 - \frac1N \log (a_{1l} \cdots a_{kl}) + c_0 \\
&\le \frac{1+t}N \log(a_{1l}\cdots a_{nl}) +\left (1-\frac{n(1+t)}N\right) \log a_{nl}
-\log\frac\delta2 - \frac1N \log(a_{1l}\cdots a_{kl})  + c_0. \nonumber
\end{align}
Since the $a_{k+1,l}, \ldots, a_{n,l}$ are bounded (for all $l$) from below (by $\varepsilon_0$ ) and above (by $M_0$), the last expression in \eqref{z0z} is bounded from above by
$\frac tN \log(a_{1l}\cdots a_{kl})$ plus a quantity independent of $l$. But, since by hypothesis
$a_{1l}, \dots, a_{kl} \to 0$ as $l\to\infty$, obviously
as $l\to \infty$ the last quantity in \eqref{z0z} tends to $-\infty$.
This establishes the desired result for the case where $q\in [1,n]$.

Now suppose $q\in (0,1)$. From \eqref{partition1} we see that
\[
\lim_{\delta\to 0^{+}}\alpha_1 = \lim_{\delta\to 0^{+}} \frac{\mu(\Omega_{1,\delta})}{|\mu|}
= \frac{\mu(S^{n-1} \setminus S^{n-2})}{|\mu|} = 1- \frac{\mu(\sn\cap \xi_{n-1})}{|\mu|}>0,
\]
where the last inequality follows from \eqref{smi-1}. 
Choose $\delta>0$ so that
$\alpha_1 = \alpha_1 (\delta)>0$, and since $\alpha_1 + \cdots + \alpha_n=1$, we have
\[
\text{ $\alpha_1>0$\quad and\quad $\alpha_2 + \cdots + \alpha_n < 1$.}
\]
From Lemmas \ref{dmp-lem3} and  \ref{dmp-lem1} we have, for sufficiently large $l$,
\begin{align*}
-\frac{1}{|\mu|} \int_{\sn} &\log \, \sum_{i=1}^n \frac{|u\cdot e_{il}|}{a_{il}}\, d\mu(u)
+\frac1q \log \int_{\sn} \left(\sum_{i=1}^n a_{il} |u\cdot e_{il}|\right)^{-q} \, du \\
&\le \sum_{i=1}^n \alpha_i \log a_{il} -\log \frac\delta2 + c_0,
\end{align*}
where $c_0$ is a constant independent of the sequences $a_{1l}, \ldots, a_{nl}$.

Since $\alpha_1$ is positive,
and $a_{1l} \to 0^+$ as $l\to\infty$,
we have $\alpha_1 \log a_{1l} \to -\infty$. This and the assumption that
$a_{1l}, \ldots, a_{nl}$ are bounded from above, allows
us to conclude that, as $l\to \infty$,
\[
\sum_{i=1}^n \alpha_i \log a_{il}\  \to\  -\infty.
\]
This establishes the desired result for the case where $q\in (0,1)$.
\end{proof}

\subsection{Existence of a solution to the maximization problem}

The following lemma establishes existence of solutions to the maximization problem
associated with the dual Minkowski problem.

\begin{lemm}\label{dmp-lem5}
Suppose $q\in (0,n]$, and $\mu$ is a non-zero, finite Borel measure on $\sn$. Suppose that the functional $\Phi:\kne \to \rbo$ is defined for $Q\in\kne$ by
\[
\Phi(Q) = \frac1{|\mu|} \int_{\sn} \log \rho_Q(u)\, d\mu(u)
+\frac1q \log \int_{\sn} h_Q(u)^{-q}\, du.
\]
If $\mu$ satisfies the $q$-th subspace mass inequality, then there exists a $K\in\kne$ so that
\[
\sup\nolimits_{Q\in\kne} \Phi(Q) = \Phi(K).
\]
\end{lemm}

\begin{proof}
Let $Q_l$ be a maximizing sequence of origin-symmetric convex bodies; i.e.
\[
\lim\nolimits_{l\to \infty}\Phi(Q_l) = \sup\nolimits_{Q\in\kne} \Phi(Q).
\]
Since $\Phi(\lambda Q) = \Phi(Q)$ for $\lambda>0$, we can assume that
the diameter of $Q_l$ is 1. By the Blaschke selection theorem, $Q_l$ has
a convergent subsequence, denoted again by $Q_l$, with limit $K$, which must be
an origin-symmetric compact convex set.
We will prove that $K$ is not degenerate; i.e., $K$ has non-empty interior.

Let $E_l\in\kne$ be the John ellipsoid associated with $Q_l$, that is, the ellipsoid of maximal volume
contained in $Q_l$. Then, as is well known (see Schneider \cite{S14}, p.\ 588),
\[
E_l \subset Q_l \subset \sqrt n E_l.
\]

For each ellipsoid $E_l$, there is a right parallelotope $P_l$ so that
\[
P_l \subset E_l \subset \sqrt n P_l.
\]
This is easily seen when $E_l$ is a ball.
The general case can easily be seen as follows:
Transform $E_l$ into ball using
an affine transformation whose eigenvectors are along the principal axes of $E_l$ and whose eigenvalues are chosen so that $E_l$ is transformed into a ball.
Therefore,
\begin{equation}\label{s5.1}
P_l \subset Q_l \subset n P_l.
\end{equation}
But this means that
\begin{equation}\label{s555}
\text{$\rho_{Q_l} \le \rho_{n P_l} = n\rho_{P_l}$\quad and \quad $h_{P_l} \le h_{Q_l}$.}
\end{equation}

The support function of the right parallelotope $P_l$ can be written, for $u\in\sn$, as
\begin{equation}\label{s5.3}
h_{P_l}(u) = \sum_{i=1}^n a_{il} |u\cdot e_{il}|,
\end{equation}
where the orthonormal basis $e_{1l},\ldots.e_{nl}$ is ordered so that
$0<a_{1l} \le \cdots \le a_{nl}$. The radial function of $P_l$, for $u\in\sn$, is given by
\[
\rho_{P_l}(u) = \min\nolimits_{1\le i \le n} \frac{a_{il}}{|u\cdot e_{il}|}.
\]
Thus,
\begin{equation}\label{s5.2}
\rho_{P_l}(u) \le \left(\frac1n\,\sum_{i=1}^n  \frac{|u\cdot e_{il}|}{a_{il}}\right)^{-1}.
\end{equation}
Therefore, from \eqref{s555}, \eqref{s5.2}, and \eqref{s5.3},  we have
\begin{align}\label{dmp-lem5-1}
\Phi(Q_l) &\le \Phi(P_l) + \log n \\
&\le \frac1{|\mu|} \int_{\sn} \log \left(\sum_{i=1}^n \frac{|u\cdot e_{il}|}{a_{il}}\right)^{\negthinspace\negthinspace{-1}} d\mu(u)
+\frac1q \log \int_{\sn} \left(\sum_{i=1}^n a_{il} |u\cdot e_{il}|\right)^{\negthinspace\negthinspace{-q}} du + 2\log n. \nonumber
\end{align}

Since the diameter of each $Q_l$ is $1$, the parallelotopes $P_l$ are bounded. Using the Blaschke selection theorem, we conclude that sequence
$P_l$ has a convergent subsequence, denoted again by $P_l$, whose limit we call $P$.

Since, $Q_l \to K$, and $P_l \subset Q_l$, while $P_l \to P$, we must have $P \subset K$. Suppose that $K$ has empty interior in $\rn$. Hence,
$P$ is a degenerated right parallelotope.
Thus, there exists a $k$ such that $1\le k < n$ and a $\varepsilon_0>0$ so that $a_{1l}, \dots, a_{kl} \to 0^+$, while
$a_{k+1, l}, \dots, a_{nl} \ge \varepsilon_0 >0$, for all $l$.
Moreover, as $l\to\infty$, the orthonormal basis
$e_{1l}, \dots, e_{nl}$
converges to
an orthonormal basis
$e_{1}, \dots, e_{n}$. (This might require taking subsequences and replacing some $e_i$ by $-e_i$.)

From \eqref{dmp-lem5-1} and Lemma \ref{dmp-lem4},
$\Phi(Q_l) \to -\infty$ as $l\to \infty$. Since $Q_l$ is a maximizing sequence,
\[
\lim_{l\to \infty} \Phi(Q_l) \ge \Phi(B) = \frac1q \log(n\omega_n).
\]
Thus we have the contradiction, that shows that $K$ must have non-empty interior.
\end{proof}

\subsection{Existence of a solution to the dual Minkowski problem}

The main existence theorem for the dual Minkowski problem stated in Introduction is implied
by the following theorem.

\begin{theo}
Suppose $\mu$ is a non-zero finite even Borel measure on $\sn$ and $q\in (0,n]$.
If the measure $\mu$ satisfies the $q$-th subspace mass inequality,
then there exists an origin-symmetric convex body $K$ in $\rn$ so that
$\wt C_q(K,\cdot) = \mu.$
\end{theo}

The proof
follows directly from  Lemmas \ref{lem-max2} and  \ref{dmp-lem5}.

The theorem above shows that the subspace mass inequality is a sufficient condition
for the existence of the dual Minkowski problem. When $0<q\le 1$, the subspace mass inequality
means that the given even measure is not concentrated in any subspace of co-dimension 1.
This condition is obviously necessary.  When $q=n$, if the given even measure is not concentrated in
two complementary subspaces, it was proved in \cite{BLYZ13jams} that the subspace mass inequality
is also necessary. Progress regarding the intermediate cases $1<q<n$ would be most welcome.

\section*{Acknowledgement} The authors thank Yiming Zhao for his comments on various versions of this work.

\end{document}